\documentclass[reqno,10pt]{amsart}
\usepackage[T1]{fontenc}
\usepackage{graphicx}
\usepackage{mathtools}
\usepackage{amssymb}
\usepackage{amsthm}
\usepackage{thmtools}
\usepackage{mathrsfs}
\usepackage[mathscr]{euscript}
\usepackage{bigints}
\usepackage[dvipsnames]{xcolor}
\usepackage{tikz-cd}
\usepackage{nameref}
\usepackage{ dsfont }
\usepackage{hyperref}
\usepackage {amssymb,latexsym,amsthm,amsmath,amsfonts, mathtools, multirow, longtable}
\usepackage{enumitem,color}
\usepackage{amsmath,tikz-cd}

\usepackage[left=2cm, right=2cm, top=2cm, bottom=2cm]{geometry}

\usepackage{amsmath,amssymb,latexsym}

\usepackage{hyperref}
\usepackage{color}
\usepackage[all]{xy}
\hypersetup{
	colorlinks=true, 
	linktoc=all,     
	linkcolor=blue,  
}

\newcommand{\zp}{\mathbb{Z}_{p}}
\newcommand{\qp}{\mathbb{Q}_{p}}

\newcommand{\q}{\mathbb{Q}}
\newcommand{\z}{\mathbb{Z}}

\newcommand{\A}{\mathbb{A}}
\newcommand{\h}{\mathbf{h}}
\newcommand{\I}{\mathbb{I}}
\newcommand{\f}{\mathbf{f}}
\newcommand{\T}{\mathbf{T}}
\newcommand{\mcL}{\mathcal{L}}
\newcommand{\mcO}{\mathcal{O}}
\newcommand{\mcA}{\mathcal{A}}
\newcommand{\mck}{\mathcal{K}}
\newcommand{\ma}{\mathfrak{a}}
\newcommand{\p}{\mathfrak{p}}

\newcommand{\n}{\mathfrak{n}}
\newcommand{\ml}{\mathfrak{l}}
\newcommand{\ip}{\I_P}

\newcommand{\kinf}{K_\infty}
\newcommand{\kip}{K_\infty^+}
\newcommand{\ikinf}{\I_{K_\infty}}
\newcommand{\ikip}{\I_{K_\infty^+}}
\newcommand{\vp}{\varpi}
\newcommand{\ox}{\overline{x}}
\newcommand{\oy}{\overline{y}}
\newcommand{\orhof}{\overline{\rho}_\f}
\newcommand{\wx}{\widetilde{x}}
\newcommand{\wy}{\widetilde{y}}

\newcommand{\tp}{\widetilde{P}}
\newcommand{\wq}{\widetilde{Q}}
\newcommand{\tfp}{\mathbf{T}_{\mathbf{f},P}}
\newcommand{\Varpi}{\rotatebox[origin=c]{180}{$\Pi\kern-0.361em\Pi$}}

\theoremstyle{plain}
\newtheorem{thm}{Theorem}[section]
\newtheorem{lem}[thm]{Lemma}
\newtheorem{prop}[thm]{Proposition}
\newtheorem{cor}[thm]{Corollary}

\newtheorem{rem}[thm]{Remark}

\theoremstyle{definition}
\newtheorem{defn}{Definition}

\newtheorem*{acknowledgements}{Acknowledgments}

\title{A $p$-Converse theorem for Real Quadratic Fields}
\begin{document}

\author{Muskan Bansal$^{1,2}$} 
  \address{$^1$ Harish-Chandra Research Institute, Chhatnag Road, Jhunsi, Prayagraj - 211019, India}
  \address{$^2$Homi Bhabha National Institute, Training School Complex, Anushakti Nagar, Mumbai 400094, India}
\email{}

    \author{Somnath Jha$^3$} 
    \address{$^3$ Depratment of Mathematics, Indian Institute of Technology Kanpur,  Kanpur - 208016, India}
	\email{}
    
\author{Aprameyo Pal$^{1,2}$} 
   
	\email{}

    \author{Guhan Venkat$^4$} 
    \address{$^4$ Department of Mathematics, Ashoka University, Rajiv Gandhi Education City, Haryana - 131029,
    India.}
	\email{muskanbansal@hri.res.in, jhasom@iitk.ac.in, aprameyopal@hri.res.in, guhan.venkat@ashoka.edu.in}


\begin{abstract}
    Let $E$ be an elliptic curve defined over a real quadratic field $F$. Let $p > 5$ be a rational prime that is inert in $F$ and assume that $E$ has split multiplicative reduction at the prime $\p$ of $F$ dividing $p$. Let $\Varpi(E/F)$ denote the Tate-Shafarevich group of $E$ over $F$ and $ L(E/F,s) $ be the Hasse-Weil complex $L$-function of $E$ over $F$. In this setting, we establish a rank 1 $p$-converse theorem to Zhang's generalisation of Gross-Zagier and Kolyvagin theorem. More precisely, under some technical assumptions, we show that 
    \begin{align*}
        \mathrm{rank}_\z \hspace{0.01mm} \hspace{1mm} E(F) = 1 \text{ and } \hspace{1mm} \#(\Varpi(E/F)_{p^\infty}) < \infty \implies \mathrm{ord}_{s=1} \ L(E/F,s) = 1. 
    \end{align*}
We also indicate an application to a  $p$-converse to Gross-Zagier and Kolyvagin theorem in a parallel setting over $\q$. 
\end{abstract}

\maketitle
\tableofcontents

\section{Introduction}

Let $L$ be a number field and $E$ be an elliptic curve defined over $L$. Denote by $L(E/L,s)$ the Hasse-Weil complex $L$-function of $E$ over $L$ and by $\Varpi(E/L)$ the Tate-Shafarevich group of $E$ over $L$. It is expected that $L(E/L,s)$ can be analytically continued to all of $\mathbb C$  and it will satisfy a functional equation relating its values at $s$ and $2-s$. Further, a part of the Birch and Swinnerton-Dyer conjecture states that $\mathrm{rank}_\z \hspace{0.01mm} \hspace{1mm} E(L) = \mathrm{ord}_{s=1} L(E/L,s)$ and that $\Varpi(E/L)$ is finite. For $L=\q$, combining the celebrated works of Gross-Zagier   \cite{GrossZagier} and Kolyvagin \cite{Kolyvagin}, one has the following result.   
\begin{thm}\label{GZK}
   Let $E$ be an elliptic curve defined over $\q$. Then, for $r\in \{0,1\},$ 
   \begin{align*} \mathrm{ord}_{s=1} L(E/\q,s) =r \implies \mathrm{rank}_\z \hspace{0.01mm} \hspace{1mm} E(\q) = r \text{ and } \Varpi(E/\q) \text{ is finite.}\end{align*}
\end{thm}

For a totally real number field $L$, under suitable hypotheses, an analogue of Theorem \ref{GZK} has been established by  Zhang \cite{Zhang} (see \cite[Theorem A]{Zhang} for a precise statement).

Now fix once and for all, a real quadratic field $F$ of discriminant $D_F$. As $F$ is a real quadratic field, it is known due to \cite{FLHS} that an elliptic curve $E$ defined over $F$ is modular and hence $L(E/F,s)$ can be analytically continued to the entire complex plane. Further, in the above setting, the following result is a special case of \cite[Theorem A]{Zhang}:

\begin{thm}\label{ZhangMainResult}
  Let $E$ be an elliptic curve defined over a real quadratic field $F$. Let $\mathfrak q$ be a (finite) prime of $F$ and  assume that the conductor of $E/F$, $\n_E = \n\mathfrak q$ with $\mathfrak q \nmid \n$. Then, for $r\in \{0,1\},$ 
   \begin{align*} \mathrm{ord}_{s=1} L(E/F,s) =r \implies \mathrm{rank}_\z \hspace{0.01mm} \hspace{1mm} E(F) = r \text{ and } \Varpi(E/F) \text{ is finite.}\end{align*}
\end{thm}

Throughout, we also fix a rational prime $p>5$ which is inert in $F$ and let $\p$ be the prime in $F$ dividing $p$. Let $E$ be an elliptic curve defined over $F$ of conductor $\n_E$. We assume that $\n_E = \n\p, \p \nmid \n$ and $E$ has split multiplicative reduction at the prime $\p$. We aim to prove for the rank $1$ i.e. $r=1$ case, a "$p$-converse" of Theorem \ref{ZhangMainResult} when $E$ has split multiplicative at the inert prime $\p \mid p$ in $F$. More precisely, let \textbf{(irred)} and \textbf{(MML)} be as in Section \ref{BigGaloisRepresentation} and Theorem \ref{WanMainResult}, respectively. We prove the following:
\begin{thm} \label{Our Main Thm} (Theorem \ref{main theorem in text})
     Let $p > 5$ be a prime and $F$ be a real quadratic field such that $p$ is inert in $F$. Denote by $\p$, the prime of $F$ lying over $p$. Let $E$ be an elliptic curve defined over $F$ having split multiplicative reduction at $\p$. Suppose \textbf{(irred)} and \textbf{(MML)} hold. Furthermore, assume that 
    \begin{enumerate}
        \item $\mathrm{rank}_\z \hspace{0.01mm} \hspace{1mm} E(F) = 1$,

        \item $\Varpi(E/F)_{p^\infty}$ is finite.
    \end{enumerate}
    Then $\mathrm{ord}_{s=1} L(E/F,s) = 1$.
\end{thm}

As a result, we can conclude using Theorem \ref{ZhangMainResult} that

\begin{thm}
    Let $p > 5$ be a prime and $F$ be a real quadratic field such that $p$ is inert in $F$. Denote by $\p$, the prime of $F$ lying over $p$. Let $E/F$ be an elliptic curve having split multiplicative reduction at $\p$. Assume that \textbf{(irred)} and \textbf{(MML)} hold. Then
    \begin{align*}
        \mathrm{rank}_\z \hspace{0.01mm} \hspace{1mm} E(F) = 1 \text { and } \hspace{1mm} \# \ \Varpi(E/F) < \infty \iff \mathrm{ord}_{s=1} L(E/F,s) = 1.
    \end{align*}
\end{thm}

Moreover, as a consequence of Theorem \ref{Our Main Thm}, we can prove a $p$-converse theorem for an elliptic curve defined over $\q$ by removing hypothesis 2 in \cite[Theorem A]{Rodolfo}. More precisely, we show that

\begin{thm} \label{p-converse over Q} (Theorem \ref{proof over Q})
    Let $p > 5$ be a rational prime. Let $E/\q$ be an elliptic curve of conductor $N_E := Np$. Assume that $E$ has split multiplicative reduction at $p$. Suppose \textbf{(irred)} holds. Further, assume that
    \begin{enumerate}
        \item $\mathrm{rank}_\z \hspace{0.1mm} E(\q) = 1$,
        \item $\Varpi(E/\q)_{p^\infty}$ is finite.
    \end{enumerate}
    Then $\mathrm{ord}_{s=1} \hspace{0.1mm} L(E/\q,s) = 1$.
\end{thm}




        
\begin{rem}
  We remark that, in fact, our results (Theorems \ref{Our Main Thm}, \ref{p-converse over Q}) hold for $p = 5$, if we exclude the following case: the projective image $\overline{G}$ of $\overline{\rho}$ is isomorphic to $\mathrm{PGL}_2(\mathbb{F}_p)$ and the mod-$p$ cyclotomic character factors through $G_F \rightarrow \overline{G}^{ab} \cong \z/{2\z}$. Note that for $p=5$, this exclusion is imposed due to a result of Fujiwara \cite[Theorem 8]{Wan} which was used in \cite{Wan}. In order to apply Wan's result (Theorem \ref{WanMainResult}) in our proof, we need to make this exclusion for $p=5$.
\end{rem}
\subsection{Previous work} There has been a considerable interest in establishing various   `$p$-converse' results  to the theorems of Gross-Zagier and Kolyvagin (Theorem \ref{GZK}) in the literature in the recent past. The choice of the prime $p$  and the reduction type of the  elliptic curves  at $p$ play an important role in these results. When $r=0$, i.e., in the rank $0$ case, $p$-converse theorems to Theorem \ref{GZK} are related to  one sided divisibility in the Iwasawa Main Conjecture (cf. \cite{SkinnerUrban} and \cite{SkinnerRank0}). In the rank $1$, i.e., $r=1$ case, for a non-CM elliptic curve $E/\q$, important results on $p$-converse theorems were established independently, around the same time, by \cite{SkinnerPConverse} and W. Zhang \cite{ZhangPConverse}, under some  hypotheses on the conductor of $E/\q$. The prime $p$ involved in both the cases, was an ordinary prime for the elliptic curve and  Castella-Wan extends $p$-converse theorem  for a prime $p$ of  (good) supersingular reduction of $E$ in \cite{CastellaWanPConverse}.  \cite{KimPConverse} proves a "soft" $p$-converse assuming isomorphism of a certain local restriction map at the prime $p$. 
For an Eisenstein prime $p$, a $p$-converse result appears in \cite{CGLS}. For a prime $p$ of multiplicative reduction of $E/\q$, a $p$-converse result in $r=1 $ case has been proven in \cite{skinner2014indivisibilityheegnerpointsmultiplicative} and \cite{castella2024exceptionalzerosheegnerpoints}, under certain hypotheses. Moreover, under a different set of hypotheses, Venerucci \cite{Rodolfo} has also proven a $p$-converse theorem in the rank $1$ case, for a prime $p$ of split multiplicative reduction of an elliptic curve  $E/\q$. The approach for proving a $p$-converse theorem taken for a CM elliptic curve over $\q$ is distinct from that of a non-CM elliptic curve and important results in this setting are due to Rubin,  Burungale-Tian \cite{BurungaleTian}, and many others. We refer to \cite{pConverseSurvey} for a comprehensive survey of literature on  $p$-converse theorems. More recently, a rank zero $p$-converse result  has been proven in \cite{BurungaleTian2} for an elliptic curve defined over an imaginary quadratic field. Furthermore, rank-1 $p$-converse results under different sets of  hypotheses have appeared in \cite{Skinner2025} and \cite{Skinner2026}.
\\

\noindent However, almost all of these results are valid for elliptic curves defined over $\q$. In this article, we  extend the strategy of \cite{Rodolfo} to prove for the first time a similar result beyond $\q$, or in other words, we discuss a "$p$-converse" result for a real quadratic field, of the rank $1$ case of Theorem \ref{ZhangMainResult}. Moreover, for  $p > 5$, in Section \ref{p converse over Q}, we strengthen the $p$-converse result   over $\q$ in  \cite{Rodolfo} by removing a technical assumption (hypothesis 2) in \cite[Theorem A]{Rodolfo}.  

\subsection{Outline  of the proof:}
We have employed various new ingredients to establish our $p$-converse results - for example, a two-variable $p$-adic $L$-function constructed by \cite{MokExceptionalZero}, a central derivative formula in \cite{MokHeegnerPoints}, a three-variable $p$-adic $L$-function and the main conjecture proven in \cite{Wan}, the interpolation formula for $p$-adic $L$-functions for Hilbert modular forms \cite{BergdallHansen}. We explain the usage of some of these in the rest of this subsection.
\\

Since $F$ is a  real quadratic field, there is a cuspidal Hilbert newform $f := f_E$ of parallel weight 2 over $F$, with trivial Nebentypus attached to $E/F$ by the modularity theorem in \cite{FLHS}. Let $\f$ be the Hida family that passes through $f$ (for details, see Section \ref{Hida family}).
We begin by mentioning the main result of \cite{Wan}. 
\\
Let $K$ be a totally imaginary quadratic extension of $F$ such that $\p$ splits in $K$ and let $K_\infty$ be a $\zp^3$-extension of $K$ (for details, see Section \ref{K_infty}).
\\
Under the hypotheses in Section \ref{WanMainResult}, Wan constructs a $p$-adic $L$-function $\mathcal{L}^S_K(\f)$, where $S$ is a finite set of non-archimedean primes containing all the bad primes of $E/F$. Let $ \mathrm{char}^S_{K_\infty}(\f)$ be the characteristic ideal attached to the Pontryagin dual of the Selmer group $Sel_{\kinf}^S(\f)$ as defined in Section \ref{Wan Sel Gp}. Then in \cite[Theorem 101]{Wan}, it is proven that
\begin{align*}
     \mathrm{char}^S_{K_\infty}(\f) \subset (\mathcal{L}^S_K(\f)).
\end{align*}
Using the Iwasawa main conjecture of Wan, in Corollary \ref{ineq2} we deduce the inequality
\begin{align} \label{introineq1}
    \mathrm{ord}_{k=2} L_p^{cc}(f_\infty/K,k) \leq \mathrm{len}_{P}(X_{F_\infty}^{S,cc}(\f/K)).
\end{align}
Here the notations are as follows. Let $\I$ be a local domain which is finite over $\Lambda$ (for details, see Section \ref{Hida family}). $P$ is the prime of $\I$ which specialises the Hida family $\f$ to $f$ and $L_p^{cc}(f_\infty/K,k)$ is defined using $\mathcal{L}_K^S(\f)$ (for details, see Section \ref{Central Critical Def}). As shown in (\ref{Central Critical in Terms of Mok}), $L_p^{cc}(f_\infty/K,k) = \lambda L_p(f_\infty,1,k,k/2) L_p(f_\infty,\epsilon_K,k,k/2)$ is the product (up to an element $\lambda \in \overline{\qp}^*$) of two-variable $p$-adic $L$-functions defined in \cite{MokExceptionalZero} evaluated at the trivial character 1 and at the quadratic character $\epsilon_K$, which is the quadratic character associated with $K$. $\mathrm{len}_P(\star)$ denotes the length of $\star_P$ as an $\ip$-module. $X_{F_\infty}^{S,cc}(\f/K)$ is the Pontryagin dual of the following Selmer group:
\begin{align*}
    Sel^{S,cc}_{F_\infty}(\f/K) := ker \left(H^1(G_{K,S_K},\T_\f \otimes_\I \I^*) \xrightarrow{\prod_{w|\p}p_{w_*}^-\circ \hspace{0.001 mm} res_w} \prod_{w|\p} H^1(I_w,\T_{\f,w}^- \otimes_\I \I^*) \right),
\end{align*}
where $S_K$ is the set of primes of $K$ lying above the primes in $S$.
Here $\T_\f$ is the 2-dimensional self-dual representation associated with the Hida family $\f$, $\T_{\f,w}^+$ is a rank 1 $\I$-submodule. They fit into the following exact sequence
\begin{align*}
    0 \rightarrow \T_{\f,w}^+ \rightarrow \T_\f \rightarrow \T_{\f,w}^- \rightarrow 0.
\end{align*}
For more details, see Section \ref{BigGaloisRepresentation}.
\\
Next, in \cite{MokHeegnerPoints}, Mok proved that for $\psi \in \{1, \epsilon_K\}$, the function $L_p(f_\infty,\psi,k,k/2)$ vanishes to order at least 2 at $k = 2$ and that there exists an element $\mathbb{P}_\psi \in (E(K) \otimes \q)_\psi$ and $l \in \q^*$ such that
\begin{align*}
    \frac{d^2}{dk^2} L_p(f_\infty,\psi,k,k/2) \vline_{k=2} = l \cdot (\log\mathrm{Norm}_{E/F_\p}(\mathbb{P}_\psi))^2,
\end{align*}
using which we show in Corollary \ref{ineq1} that under the assumption that the global root number $w(E/K)$ is $-1$, we have that 
\begin{align} \label{introineq2}
    \mathrm{ord}_{k=2} L_p^{cc}(f_\infty/K,k) \geq 4
\end{align}
and moreover,
\begin{align} \label{introineq3}
    \mathrm{ord}_{k=2} L_p^{cc}(f_\infty/K,k) = 4 \iff \mathrm{ord}_{s=1}L(E/K,s) = 2,
\end{align}
where $L(E/K,s) $ is the Hasse-Weil complex $L$-function of $E$ over $K$ and one has the factorisation  $L(E/K,s)=L(E/F,s)L(E/F,\epsilon_K,s)$, i.e., $L(E/K,s)$ is  the product of the complex $L$-function of $E/F$ with the complex $L$-function of $E/F$ twisted by $\epsilon_K$.
\\
Now, define the strict Selmer group of $\f/K$ as
\begin{align*}
    Sel^{cc}_{Gr}(\f/K) := ker \left(H^1(G_{K,S},\T_\f \otimes_\I \I^*) \rightarrow \prod_{w|\mathfrak{p}} H^1(K_w,\T_{\f,w}^- \otimes_\I \I^*) \right)
\end{align*}
and let $X^{cc}_{Gr}(\f/K)$ denote its Pontryagin dual. Using \cite{Nekovar}, we define an alternating, $Gal(K/F)$-equivariant, bilinear pairing $\langle -,- \rangle^{Nek}_K$ on Selmer complexes. Then using some Kummer theory and this pairing, we conclude by Theorem \ref{ineq4} that 
\begin{align} \label{introineq4}
    \mathrm{len}_P(X^{cc}_{Gr}(\f/K)) = 2.
\end{align}
The calculations in Lemma \ref{ineq3}, give us the following inequality
\begin{align} \label{introineq5}
    \mathrm{len}_P(X^{S,cc}_{F_\infty}(\f/K)) \leq \mathrm{len}_P(X^{cc}_{Gr}(\f/K)) + 2.
\end{align}
Inequalities (\ref{introineq1}),(\ref{introineq2}),(\ref{introineq4}) and (\ref{introineq5}) allow us to conclude that
\begin{align*}
     \mathrm{ord}_{k=2} L_p^{cc}(f_\infty/K,k) = 4
\end{align*}
and hence by (\ref{introineq3}), we obtain that
\begin{align*}
    \mathrm{ord}_{s=1}L(E/K,s) = 2.
\end{align*}
Now by the choice of the field $K$ that we make in Section \ref{ChoiceOfK}, we get our result that
\begin{align*}
    \mathrm{ord}_{s=1}L(E/F,s) = 1.
\end{align*}

\subsection{Future Outlook}
Assuming the modularity of the elliptic curve $E/L$ and some other technical conditions, one can hope to extend these techniques to a general totally real number field $L$ for which Leopoldt's conjecture holds true (in particular, for any abelian totally real number field). We also indicate some possible extensions of our main results in Remark \ref{possible-extensions}.

However, note that for a CM elliptic curve defined over a real quadratic field, one would expect that the techniques required to prove a $p$-converse result would be rather different.

 \begin{acknowledgements}
	A. Pal and S. Jha acknowledge ANRF/ARG/2025/005885/MS grant. S. Jha is also supported by ANRF/ARGM/2025/002127/MTR grant. M. Bansal acknowledges the support of the HRI institute fellowship for PhD students.
\\
 The authors would like to thank Haruzo Hida, C.~S.~Rajan and Sudhanshu Shekhar for their valuable discussions and encouragement. They are also grateful to Francesc Castella for suggesting an application to the corresponding  \( p \)-converse result over \( \mathbb{Q} \). Finally, the authors would also like to thank the anonymous referee for a careful reading and her/his various suggestions for the improvement of the manuscript.

    \end{acknowledgements}

\section{Hida Theory}
We start with some preliminaries on the Hida theory following \cite{Hida1} and \cite{MokExceptionalZero}. We also recall the big Galois representation associated to the Hida family following \cite{Nekovar}.
\\
Recall we have  fixed a real quadratic field $F$ and for every place $v$ of $F$, fix an embedding $i_v : \overline{F} \hookrightarrow \overline{F_v}$ and a decomposition group $i_v^* : G_{F_v} \hookrightarrow G_F$ at $v$. Denote by $I$ the set of the two embeddings of $F$ into $\mathbb{R}$. Let $t = \sum_{\sigma \in I} \sigma \in \z[I]$. Let $E/F$ be an elliptic curve as in Theorem \ref{Our Main Thm}.

\subsection{Hecke Algebra} \label{HeckeAlgebra}
Let $\mcO$ be a finite extension of $\zp$ and for every $k \geq 2$, let $h_{k}^{\mathrm{ord}}(\n\p^\infty,\mcO)$ be the $p$-ordinary Hecke algebra whose coefficients lie in $\mcO$, where $\n\p$ is the conductor of $E/F$. Let $Cl_F(\n\p^\alpha)$ be the narrow ray class group of $F$ modulo $\n\p^\alpha$ for $\alpha \geq 1$. Let $Z_F(\n) = \varprojlim_{\alpha} Cl_F(\n\p^\alpha)$. Then we can decompose $Z_F(\n) = W_F \times Z_F(\n)_{tor}$, where $W_F$ is $\zp$-free and $Z_F(\n)_{tor}$ is finite. Let $Z_\alpha$ denote the kernel of the natural projection $Z_F(\n) \rightarrow Cl_F(\n\p^\alpha)$ and let $W_\alpha := W_F \cap Z_\alpha$. Define the following Iwasawa algebras
\begin{align*}
    \Lambda := \varprojlim_{\alpha} \mcO[W_F/W_\alpha], \hspace{5mm} \mcA := \varprojlim_{\alpha} \mcO[Cl_F(\n\p^\alpha)].
\end{align*}
Let $\q_\infty$ be the cyclotomic $\zp$-extension of $\q$. Write $\Gamma$ for the image of $Gal(F\q_\infty/F) \hookrightarrow Gal(\q_\infty/\q) \rightarrow 1+p\zp$ (note that $\Gamma \cong \zp$) and fix its topological generator $\gamma$. We have $\Lambda \cong \mcO[[\Gamma]]$ and $\mcA = \Lambda[Z_F(\n)_{tor}]$. If $\ml$ is an ideal prime to $\n\p$, write $[\ml]$ for the group ring element of $\mcA$ and $\langle [\ml] \rangle$ for the corresponding element in $\Lambda$.
\\
Similarly, define $Z_\q(1), Z_\q(1)_{tor}$ and $W_\q$. Since $p > 5$, one can see that $Z_\q(1) = \zp^\times, Z_\q(1)_{tor} = \mathbb{F}_p^\times$ and $W_\q = 1+p\zp$. Denote by $\omega_\q : Z_\q(1) \rightarrow Z_\q(1)_{tor}$ the Teichmuller character and by $\langle \cdot \rangle_\q : Z_\q(1) \rightarrow W_\q$ the natural projection. Recall the norm map, $Nm : Z_F(1) \rightarrow Z_\q(1)$ and set $\omega := \omega_F := \omega_\q \circ Nm$ and $\langle \cdot \rangle := \langle \cdot \rangle_F := \langle \cdot \rangle_\q \circ Nm$. By abuse of notation, we denote by the same symbols the composition of $\omega$ and $\langle \cdot \rangle$ with the projection $Z_F(\n) \rightarrow Z_F(1)$.
\\
In \cite[Theorem 3.2]{Hida1}, Hida proves the following isomorphism for every $k,k' \geq 2$
\begin{align*}
    h_{k}^{\mathrm{ord}}(\n\p^\infty,\mcO) \xrightarrow{\sim} h_{k'}^{\mathrm{ord}}(\n\p^\infty,\mcO).
\end{align*}
So we can write $\h^{\mathrm{ord}}(\n,\mcO)$ for $h_{k}^{\mathrm{ord}}(\n\p^\infty,\mcO).$ For every $\alpha \geq 1$, consider the morphisms $Z_F(\n)/Z_{F,\alpha}(\n) \rightarrow h_k(\n \p^\alpha, \mcO)$ sending the class of $\ml$ to the diamond operator $[\ml]_k$. These morphisms are compatible with each other for every $\alpha$. So we can define the morphism on the inverse limits
\begin{align*}
    [\cdot]_{k,\infty} : Z_F(\n) \rightarrow \h^{\mathrm{ord}}(\n,\mcO).
\end{align*}
This is a continuous character and hence we can extend it to $\mcA$. This defines an action of $\mcA$ on $\h^{\mathrm{ord}}(\n,\mcO)$ and hence, in particular we get a $\Lambda$-algebra structure on $\h^{\mathrm{ord}}(\n,\mcO)$. As in \cite[Section 4.1]{MokExceptionalZero}, we consider the $\Lambda$-algebra structure on $\h^{\mathrm{ord}}(\n,\mcO)$ which is defined as the twist of this $\Lambda$-algebra structure by the character $\Lambda \rightarrow \mcO$ that takes $\langle [\ml] \rangle$ to $\langle \ml \rangle^{k-2}$ for all $\ml$ coprime to $\n\p$.

\subsection{Hida Family} \label{Hida family}
Let $\lambda : W_F \rightarrow \overline{\qp}^\times$ be a continuous character. Then we can extend $\lambda$ to an algebra homomorphism (again denoted $\lambda), \lambda : \Lambda \rightarrow \overline{\qp}$. Let $P_\lambda$ denote the point of $Spec(\Lambda)_{/\mathcal{O}}(\overline{\qp})$ corresponding to $\lambda$. Define a character 
\begin{align*}
    \chi & : Z_F(\n) \rightarrow \zp^* \\
    \chi([a]) & := N_{F/\q}(a)
\end{align*}
and for each $\xi \in \z \cdot t$ put $\chi_\xi = \chi^{[\xi]} : Z_F(\n) \rightarrow \zp^*$, where $\xi = [\xi] \cdot t$. Now for each finite order character $\epsilon : W_F \rightarrow \overline{\qp}^\times$ and $m \in \z \cdot t$, write $P_{m,\epsilon}$ for $P_{\chi_m\epsilon}$. Then we have the following 
\begin{thm}
    Let $k \geq 2$ and let $\epsilon : W_F/W_\alpha \rightarrow \mathcal{O}^*$ be a finite order character. Suppose $W_\alpha = Z_\alpha$. Write $P$ for $P_{k,\epsilon}$. Then there is a canonical isomorphism of $\Lambda$-algebras
    \begin{align*}
        \h^{\mathrm{ord}}(\n,\mathcal{O})/P \cong h_{k}^{\mathrm{ord}}(\n\p^\alpha,\epsilon,\mathcal{O}).
    \end{align*}
\end{thm} 
\begin{proof}
    See \cite[Corollary 4.21]{HidaBook}.
\end{proof}
Denote by $\mcL$ the fraction field of $\Lambda$. Fix an algebraic closure $\overline{\mcL}$ of $\mcL$ and consider $\overline{\qp}$ as a subfield of $\overline{\mcL}$. Fix a $\Lambda$-algebra homomorphism $\lambda : \h^{\mathrm{ord}}(\n,\mcO) \rightarrow \overline{\mcL}$. Denote by $\mck$ the fraction field of the image of $\lambda$ inside $\overline{\mcL}$. Then $\mck/\mcL$ is a finite extension of fields. Let $\I$ be the integral closure of $\Lambda$ in $\mck$. $\I$ is a local domain that is finite over $\Lambda$. Then the image of $\lambda$ lies inside $\I$. Let 
\begin{align*}
    \mathfrak{X} := \mathfrak{X}(\I) := \mathrm{Hom}_{\mathcal{O}-alg}(\I,\overline{\qp})
\end{align*}
and put 
\begin{align*}
    \mathfrak{X}_{alg}(\I) := \{ P \in \mathfrak{X}(\I) : P|_\Lambda = P_{n,\epsilon} \hspace{1mm} \text{for some} \hspace{1mm} n \in \z \cdot t, [n] \geq 0 \hspace{1mm} and \hspace{1mm} \epsilon : W_F \rightarrow \overline{\q}^\times \}
\end{align*}
where $\epsilon$ is some finite order character. Then for every $P \in \mathfrak{X}(\I)$ we can consider the $\mathcal{O}$-algebra homomorphism $\lambda_P = P \circ \lambda : \h^{\mathrm{ord}}(N,\mathcal{O}) \rightarrow \overline{\qp}$. Then \cite[Theorem 5.6]{Hida1} states that
\begin{thm}
    There exists a unique cuspform $f_P \in S^{\mathrm{ord}}(\n,\overline{\qp})$ such that $f_P|T_0(\mathfrak{a}) = \lambda_P(T_0(\mathfrak{a}))f_P$ and the Fourier coefficients $C(\mathfrak{a},f_P)$ are equal to $\lambda_P(T_0(\mathfrak{a}))$ for all integral ideals $\mathfrak{a}$ of $F$. Conversely, suppose there is a non-zero common eigenform $f$ in $S^{\mathrm{ord}}(\n,\mathcal{O})$ of all $T_0(\mathfrak{a})$, then there exists a $\Lambda$-algebra homomorphism $\lambda : \h^{\mathrm{ord}}(\n,\mathcal{O}) \rightarrow \overline{\mathcal{L}}$ and a point $P \in Spec(\I)(\mathcal{O})$ such that $f$ is a constant multiple of $f_P$. If $f$ is a complex cuspform, then the $P$ above belongs to $\mathfrak{X}_{alg}(\I)$.
\end{thm} 
We can also write the Hida family in terms of $\I$-adic forms as follows.
\begin{defn}
    Let $\f$ be the collection $\{C(\ma,\f), C_0(\ma,\f)\}$ of elements of $\I$ where $\ma$ varies over all the integral ideals of $F$. We call $\f$ an $\I$-adic modular form of tame level $\mathfrak{m}$ if there is a subset of $\mathfrak{X}_{alg}(\I)$, Zariski dense in $\mathfrak{X}(\I)$ such that for every $P$ in it there exists a classical Hilbert modular form $f_P \in M_k(\mathfrak{m}\p^\alpha,\epsilon,Frac(\mathcal{O}))$ whose coefficients satisfy $C(\mathfrak{a},f_P) = P(C(\mathfrak{a},\f))$ and $C_0(\mathfrak{a},f_P) = P(C_0(\mathfrak{a},\f))$. 
\end{defn}
$\I$-adic modular forms form an $\I$-module, denoted $M(\mathfrak{m},\I)$ and denote by $S(\mathfrak{m},\I)$ the $\I$-submodule consisting of those forms $\f$ such that $P(\f)$ is a cusp form for every $P$. Write $M^{\mathrm{ord}}(\mathfrak{m},\I) = e \cdot M(\mathfrak{m},\I)$ and $S^{\mathrm{ord}}(\mathfrak{m},\I) = e \cdot S(\mathfrak{m},\I)$ for the space of ordinary $\I$-adic modular and cusp form, respectively. Here $e$ is Hida's ordinary projection operator. 

\subsection{The Localised Hida Family}
Let $P \in \mathfrak{X}_{alg}(\I)$ be the prime of weight 2 attached to the parallel weight 2 Hilbert modular form $f=f_E$ with trivial Nebentypus and by an abuse of notation, we write $P$ again for $ker(P)$. Then $\ip$ is a discrete valuation ring with a uniformizer $\vp := \gamma - 1$. Moreover, $\ip/P\ip \cong L$, where $L$ is a finite extension of $\qp$. $L$ is a trivial representation of $G_{K',w}$ where $K'$ is either $F$ or a quadratic extension of $F$ and $w$ is a prime of $K'$ lying above $\p$. Let $\mathcal{R} \subset \overline{\qp}[[k-2]]$ be the subring consisting of formal power series in $k-2$ which have a positive radius of convergence. As in \cite[Section 8.1]{MokExceptionalZero}, define an algebra morphism
\begin{align*}
    M : \Lambda \rightarrow \mathcal{R}
\end{align*}
which sends an element of the form $\langle[\mathfrak{l}]\rangle$ to the power series that represents the analytic function $k \mapsto \langle\mathfrak{l}\rangle^{k-2}$, where $\langle \cdot \rangle$ is defined in Section \ref{HeckeAlgebra}. This map can be extended to the map (again denoted $M$)
\begin{align*}
    M : \ip \rightarrow \mathcal{R}.
\end{align*}
Now consider all the power series $M(x)$, where $x$ varies over $\ip$. Every $M(x)$ has a positive radius of convergence. Consider the discs of convergence of these and call their intersection $U$. Since $\I$ is finite over $\Lambda$, $U$ is a $p$-adic disc with center 2. Let $\mathcal{A}(U) \subset \overline{\qp}[[k-2]]$ be the subring consisting of the formal power series converging for every $x \in U$. So, in particular, we have a morphism
\begin{align} \label{def of M}
    M : \I \rightarrow \mathcal{A}(U).
\end{align}
Note that $M(x)(2) = P(x)$ for every $x \in \I$. Now applying this map to the Fourier coefficients $C(\mathfrak{a},\f)$ of the Hida family $\f$ attached to $f$, we get a formal power series $f_\infty$ in $\mathcal{A}(U)$ with coefficients $C_\mathfrak{a}(k) := M(C(\mathfrak{a},\f))$. Moreover, for every $\kappa \in U \cap \z_{\geq 2}$, the weight-$\kappa$ specialisation of $f_\infty$ is the classical eigenform $f_\kappa$ of weight $\kappa$. In particular, $f = f_2$. We will restrict our attention to $\{\kappa \in U \cap \z_{\geq 2} : k \equiv 2 \hspace{1mm} (mod \hspace{1mm} (2(p-1)))\}$.

\subsection{Big Galois Representation} \label{BigGaloisRepresentation}
Let $\rho_\f$ be the contragredient of the big Galois representation attached to $\f$ in \cite{Hida2} with the representation space $V_\f$.
We assume the following condition throughout:
\\
\textbf{(irred)}: The residual representation $\overline{\rho}_\f$ is absolutely irreducible.
\\
Note that, in this setting,  $\mathbf{(dist)_\f}$ of \cite[Section 1.2]{Wan} is being satisfied because $E$ has split multiplicative reduction.
\\
Recall that $P \in \mathfrak{X}_{alg}(\I)$ is the prime attached to $f$ by Hida's control theorem and that $\I_P$ is a discrete valuation domain with uniformizer $\varpi$. Then following \cite[Chapter 12]{Nekovar} we have an exact sequence of $Frac(\I)[G_\p]$-modules
\begin{align*}
    0 \rightarrow V_\f^+ \rightarrow V_\f \rightarrow V_\f^- \rightarrow 0
\end{align*}
where both $V_\f^\pm$ are one-dimensional over $Frac(\I)$. $V_\f$ contains a $G_{F,S}$-invariant $\I_P$-lattice $T_\f$, where $S$ is a finite set of non-archimedean primes containing all the bad primes of $E/F$. Since the big Galois representation that we are considering is the contragredient of what is being considered in \cite{Nekovar}, the determinant representation of $T_\f$ is
\begin{align*}
    det T_\f \cong \I(\chi_{cy}\chi_\Gamma)
\end{align*}
where $\chi_{cy}$ is the $p$-adic cyclotomic character and $\chi_\Gamma$ is the character $\chi_\Gamma : G_{F,S} \twoheadrightarrow Gal(F\q_\infty/F) \rightarrow \Gamma \hookrightarrow \Lambda^*$. Put
\begin{align*}
    T_\f^+ := T_\f \cap V_\f^+, \hspace{5mm} T_\f^- := T_\f \cap V_\f^-.
\end{align*}
Both $T_\f^\pm$ are free $\I$-modules of rank 1 and we have an exact sequence of $\I[G_\p]$-modules
\begin{align*}
    0 \rightarrow T_\f^+ \rightarrow T_\f \rightarrow T_\f^- \rightarrow 0
\end{align*}
on which $G_\p$ acts as follows: let $\mathfrak{a}_\p^*$ be the unramified character which takes the Frobenius $Fr(\p)$ to $\lambda(T(\p)) =: \ma_\p$, where $T(\p)$ is the Hecke operator. Then $G_\p$ acts on $T_\f^+$ via $\ma_\p^{*-1}\chi_{cy}\chi_\Gamma$ and on $T_f^-$ via $\ma^*_\p$. So 
\begin{align*}
    T_\f^+ \cong \I(\ma_\p^{*-1}\chi_{cy}\chi_\Gamma), \hspace{5mm} T_\f^- \cong \I(\ma_\p^*).
\end{align*}
Since $\Gamma$ is uniquely 2-divisible, let
\begin{align*}
    \chi_\Gamma^{1/2} : G_{F,S} \twoheadrightarrow \Gamma \xrightarrow{1/2} \Gamma \hookrightarrow \Lambda^*
\end{align*}
be the character induced by $\chi_\Gamma$. Then using the fact that our Hilbert modular form $f$ is of parallel weight 2 and trivial character, \cite[Section 12.7]{Nekovar} constructs the following $\I[G_F]$- and respectively, $\I[G_\p]$- modules
\begin{align*}
    \T_\f := T_\f \otimes \chi_\Gamma^{-1/2}, \hspace{5mm} \T_\f^\pm := T_\f^\pm \otimes \chi_\Gamma^{-1/2}
\end{align*}
which fit in the following exact sequence
\begin{align*}
    0 \rightarrow \T_\f^+ \rightarrow \T_\f \rightarrow \T_\f^- \rightarrow 0.
\end{align*}
Moreover, we have that
\begin{align*}
    \T_\f \otimes \I_P/P\I_P & \cong V_p(E) \otimes L =: V_f
\end{align*}
where $L$ is a finite extension of $\qp$ and $V_p(E) := T_p(E) \otimes \qp$ with $T_p(E) := \varprojlim_n E(\overline{\q})[p^n]$ being the Tate module of $E/F$. Furthermore, $\T_\f$ is self-dual. More precisely, \cite[Section 12.7.12]{Nekovar} shows the existence of a skew-symmetric bilinear form
\begin{align*}
    \pi : \T_\f \otimes \T_\f \rightarrow \I(1)
\end{align*}
such that
\begin{align*}
    adj(\pi) : \T_\f \xrightarrow{\sim} \T_\f^*(1) := \mathrm{Hom}_{\I}(\T_\f,\I)(1)
\end{align*}
is an isomorphism. Moreover, we have an isomorphism of exact sequences induced by $adj(\pi)$
$$
\begin{tikzcd}
	0 \arrow[r] & 
    \T_\f^+ \arrow[r] \arrow[d,"\sim"] &
	\T_\f \arrow[r,'] \arrow[d,"\sim"] &
	\T_\f^- \arrow[r] \arrow[d,"\sim"] &
	0
	\\
	0 \arrow[r] &
	(\T_\f^-)^*(1) \arrow[r] &
	\T_\f^*(1) \arrow[r] &
	(\T_\f^+)^*(1) \arrow[r] &
	0.
\end{tikzcd}
$$
Let $X_\f \in \{T_\f,\T_\f\}$. For every prime $v$ of $\overline{\q}$ dividing $\p$, we consider $X_{\f,v}^\pm := X_\f^\pm$  as an $\I[G_v]$-module using the fixed embedding $i_v : \overline{\q} \hookrightarrow \overline{\qp}$.
\\
Let $(\cdot,\cdot)_W : V_f \times V_f \rightarrow L$ be the morphism induced by the Weil pairing on $E$.  We assume that $\pi(x \otimes y) \hspace{1mm} mod \hspace{1mm} \varpi = (x \hspace{1mm} mod \hspace{1mm} \varpi, y \hspace{1mm} mod \hspace{1mm} \varpi)_W$.

\section{Iwasawa Main Conjecture of Wan} \label{Wan}
In this section, we recall the main result proven by Wan in \cite[Theorem 3]{Wan}.
\\
Let $K$ be a totally imaginary quadratic extension of $F$ such that $\p$ splits in $K$. First, we discuss the objects in the Iwasawa main conjecture of Wan.

\subsection{Wan's $p$-adic $L$-Function} \label{K_infty}
Let $F_\infty$ be the cyclotomic $\zp$-extension of $F$ and let $\Gamma_F := Gal(F_\infty/F) \cong \zp$ denote its Galois group.  Let $K_\infty^-$ be the maximal abelian anticyclotomic $\zp$-power extension of $K$ unramified outside $p$ with Galois group $\Gamma_K^- := Gal(K_\infty^-/K) \cong \zp^2$. Denote by $K_\infty^+$ the cyclotomic $\zp$-extension of $K$, $K_\infty^+ := F_\infty K$ and its Galois group by $\Gamma_K^+ := Gal(K_\infty^+/K)$ which can be identified with $\Gamma_F$. Let $K_\infty := K_\infty^- K_\infty^+$. Then $\Gamma_K := Gal(K_\infty/K) \cong \zp^3$. Let $\I_K := \I[[\Gamma_K]]$.  Then for the Hida family $\f$ passing through $f$, in \cite[Section 7.3]{Wan}, Wan constructs a $p$-adic $L$-function $\mathcal{L}^S_K(\f) \in \I_K$ whose specialization at $\psi \in \mathfrak{X}^{arith}(\I)$ is equal (up to multiplication by a unit) to the product $L_p^S(f_\psi) \cdot L_p^S(f_\psi \otimes \epsilon_K)$, where $\epsilon_K$ is the quadratic character attached to $K/F$ and $L_p^S(f_\psi)$ and $L_p^S(f_\psi \otimes \epsilon_K)$ are the $p$-adic $L$-functions in \cite{loeffler2020iwasawatheoryquadratichilbert}. More precisely, let $\psi \in \mathfrak{X}^{arith}(\I)$ and write $\psi^{cy} : \I_K = \I[[\Gamma_K^+ \times \Gamma_K^-]] \rightarrow \overline{\qp}[[\Gamma_K^+]]$ for the morphism which is equal to $\psi$ when restricted to $\I$ and which takes $\Gamma_K^-$ to 1. Then as shown in \cite[Proposition 13.3.2]{loeffler2020iwasawatheoryquadratichilbert} there exists a unit $c$ in $\mcO[[\Gamma_F]]$ such that
\begin{align} \label{FactorisationInLZ}
   \psi^{cy}(\mathcal{L}^S_K(\f)) = c \cdot L_p^S(f_\psi) \cdot L_p^S(f_\psi \otimes \epsilon_K).
\end{align}

\subsection{Selmer Groups} \label{Wan Sel Gp}
First, we define the $S$-imprimitive Selmer groups in a general setting.
\\
Let $A$ be a profinite $\zp$-algebra and let $T$ be a free module of finite rank over $A$. Suppose $G_F$ acts continuously on $T$. Suppose for every prime we are given a free rank one $G_\p$-stable $A$-submodule $T_\p^+$ of $T$, i.e., we have an exact sequence
\begin{align*}
    0 \rightarrow T_\p^+ \rightarrow T \rightarrow T_\p^- \rightarrow 0.
\end{align*}
For any $\mathfrak{a} \in Spec(A)$ define
\begin{align*}
    Sel^S_F(T,\mathfrak{a}) := ker \left(H^1(F,T \otimes_A A^*[\mathfrak{a}]) \rightarrow \prod_{v \notin S} H^1(I_v,T \otimes_A A^*[\mathfrak{a}]) \times H^1(I_\p,T_\p^- \otimes_A A^*[\mathfrak{a}]) \right)
\end{align*}
where $I_v$ is the inertia group at $v$ for every place $v$ and $A^*$ denotes the Pontryagin dual of $A$. Now assume that $S$ contains all primes at which $T$ is ramified. Put 
\begin{align*}
    X^S_F(T,\mathfrak{a}) := \mathrm{Hom}_A(Sel^S_F(T,\mathfrak{a}),A^*[\mathfrak{a}]).
\end{align*}
Define
\begin{align*}
    Sel^S_K(T,\mathfrak{a}) := Sel^{S_K}_K(T,\mathfrak{a}), \hspace{5mm} X^S_K(T,\mathfrak{a}) := X^{S_K}_K(T,\mathfrak{a}),
\end{align*}
where $S_K$ is the set of places of $K$ over those in $S$ and if $w|\p$ then $T_w = g_wT_\p$ for $g_w \in G_F$ such that $g_w^{-1}G_{K,w}g_w \subset G_{F,\p}$.
\\
If $F''/F$ is an infinite extension, set
\begin{align*}
    Sel^S_{F''}(T,\mathfrak{a}) := \varinjlim_{F \subset F' \subset F''} Sel^S_{F'}(T,\mathfrak{a}), \hspace{5mm} X^S_{F''}(T,\mathfrak{a}) := \varprojlim_{F \subset F' \subset F''} X^S_{F'}(T,\mathfrak{a}),
\end{align*}
where $F'$ runs over the finite extensions of $F$ contained in $F''$.
\\
Now suppose that $A$ is a Krull domain and that $X$ is a finitely generated $A$-module. If $X$ is a torsion module, define the characteristic ideal of $X$ by
\begin{align*}
    \mathrm{char}_A(X) := \{ x \in A | \mathrm{ord}_Q(x) \geq \mathrm{len}_Q(X) \hspace{1mm} \text{for every height 1 prime} \hspace{1mm} Q \subset A \}.
\end{align*}
Now let $\mathcal{K}/K$ be a $\zp^r$ extension for some integer $r \geq 0$. Write $\I_\mathcal{K} := \I[[Gal(\mathcal{K}/K)]]$. Recall that we have the representation space $T_\f$ attached to the Hida family $\f$ and a free rank 1 $\I$-submodule $T_\f^+$ of $T_\f$, i.e., we have the exact sequence
\begin{align} \label{FiltrationOfSelfDualRep}
    0 \rightarrow T_\f^+ \rightarrow T_\f \rightarrow T_\f^- \rightarrow 0.
\end{align}
Now define
\begin{align*}
    T_\f(\mathcal{K}) & := T_\f \otimes_\I \I_\mathcal{K}(\epsilon_\mathcal{K}^{-1}) \\
    T_\f(\mathcal{K})^\pm & := T_\f^\pm \otimes_\I \I_\mathcal{K}(\epsilon_\mathcal{K}^{-1}),
\end{align*}
where $\epsilon_\mathcal{K} : G_K \twoheadrightarrow Gal(\mathcal{K}/K) \subset \I_\mathcal{K}^*$. Then (\ref{FiltrationOfSelfDualRep}) induces an exact sequence
\begin{align*}
    0 \rightarrow T_\f(\mathcal{K})^+ \rightarrow T_\f(\mathcal{K}) \rightarrow T_\f(\mathcal{K})^- \rightarrow 0.
\end{align*}
As $\mathcal{K}/K$ is unramified outside $p, T_\f(\mathcal{K})$ is unramified for every $v \not\in S$, so that $T_\f(\mathcal{K})$ is an $\I_\mathcal{K}[G_{F,S}]$-module. Then for any $\mathfrak{a} \in Spec(\I_\mathcal{K})$, the Selmer group $Sel^S_\mathcal{K}(\f,\mathfrak{a}) := Sel^S_\mathcal{K}(T_\f(\mathcal{K}),\ma)$ is defined and we can consider the characteristic ideal $\mathrm{char}^S_\mathcal{K}(\f) \subset \I_\mathcal{K}$ of the dual Selmer group $X^S_\mathcal{K}(\f,\mathfrak{a}) := X^S_\mathcal{K}(T_\f(\mathcal{K}),\mathfrak{a})$. When $\mathfrak{a} = 0$, write $Sel^S_\mathcal{K}(\f) := Sel^S_\mathcal{K}(\f,\mathfrak{a})$ and $X^S_\mathcal{K}(\f) := X^S_\mathcal{K}(\f,\mathfrak{a})$.

\subsection{The Main Result} 
Now we are ready to state the main result of \cite{Wan}. Recall that $p > 5$ is a rational prime. Let $L/\qp$ be a finite extension. Let $\f$ be the Hida family passing through $f$. Decompose $\mathfrak{n} = \mathfrak{n}^+\mathfrak{n}^-$ in such a way that $\mathfrak{n}^+$ (resp. $\mathfrak{n}^-$) is divisible only by the primes split (resp. inert) in $K$.
Then in \cite[Theorem 101]{Wan}, Wan proves the following.

\begin{thm} \label{WanMainResult}
    Assume the following
\begin{enumerate}
    \item $\mathfrak{n}^-$ is square-free and it has an even number of prime divisors.

    \item $\overline{\rho}_\f$ is ramified at all $v|\mathfrak{n}^-$.

    \item \textbf{(irred)}: The residual representation $\overline{\rho}_\f$ is absolutely irreducible.

    \item \textbf{(MML)} : There is a minimal modular lifting of $\overline{\rho}_\f$. 

    \item  $K$ is not contained in the narrow Hilbert class field of $F$ (i.e., there is a finite prime of $F$ which ramifies in $K$).

    \item All the primes ramified in $F/\q$ are split in $K$.
\end{enumerate}
Then we have
    \begin{align*}
        \mathrm{char}^S_{K_\infty}(\f) \subset (\mathcal{L}^S_K(\f)).
    \end{align*}
\end{thm}

\section{Central Critical $p$-adic $L$-Function, Selmer Group and a Control Theorem}
In this section, we discuss the central critical $p$-adic $L$-function and the central critical Selmer group. Following \cite{Rodolfo}, we extend his three-variable central critical $p$-adic $L$-function to the four-variable case. Using a control theorem proved in Section \ref{control theorem}, we relate the order of vanishing of this $p$-adic $L$-function with the length of the Pontryagin dual of the corresponding Selmer group.
\subsection{Central Critical $p$-adic $L$-Function} \label{Central Critical Def}
Denote by $\mathcal{A}(U \times \zp \times \zp \times \zp) \subset \overline{\qp}[[k-2,s-1,r-1,t-1]]$ the subring consisting of formal power series converging for every $(k,s,r,t) \in U \times \zp \times \zp \times \zp$. Let $\chi_{cy} : \Gamma_K^+ \cong 1+p\zp$ be the $p$-adic cyclotomic character and fix an isomorphism $\chi_{acy,1} \times \chi_{acy,2} : \Gamma_K^- \cong (1+p\zp) \times (1+p\zp)$. Recall the morphism $M$ from (\ref{def of M}). We can uniquely extend $M$ to a ring homomorphism
\begin{align*}
    \widetilde{M} : \I[[\Gamma_K^+ \times \Gamma_K^-]] & \rightarrow \mathcal{A}(U \times \zp \times \zp \times \zp) \\ 
    \widetilde{M}(\sigma) & := \chi_{cy}(\sigma)^{s-1}, \text{if} \hspace{1.5mm} \sigma \in \Gamma_K^+ \\
    \widetilde{M}(\sigma) & := \chi_{acy,1}(\sigma)^{r-1}\chi_{acy,2}(\sigma)^{t-1}, \text{if} \hspace{1.5mm} \sigma \in \Gamma_K^-.
\end{align*}
Define the $S$-primitive analytic four-variable $p$-adic $L$-function of $f_\infty/K$ as
\begin{align*}
   L_p^S(f_\infty/K,k,s,r,t) := \widetilde{M}(\mathcal{L}_K^S(\f)) \in \mathcal{A}(U \times \zp \times \zp \times \zp).
\end{align*}
Let $\mathfrak{l} \neq \mathfrak{p}$ be a prime of $F$ in $S$. Recall that we denoted the weight $k$-specialisation of $f_\infty$ by $f_k$ . Define the central critical $\mathfrak{l}$-Euler factor of $f_\infty/K$ as
\begin{align*}
    E_{\mathfrak{l}}(f_\infty/K,k) := \left(1 - \frac{C(\mathfrak{l},k)}{\langle \mathfrak{l} \rangle_F^{k/2} \omega_F(\mathfrak{l})} + \frac{\mathbf{1}_\mathfrak{n}(\mathfrak{l})}{N(\mathfrak{l})} \right) \cdot \left(1 - \frac{\epsilon_K(\mathfrak{l})C(\mathfrak{l},k)}{\langle \mathfrak{l} \rangle_F^{k/2} \omega_F(\mathfrak{l})} + \frac{\mathbf{1}_{\mathfrak{n}D_K}(\mathfrak{l})}{N(\mathfrak{l})} \right) \in \mathcal{A}(\zp),
\end{align*}
where $C(\mathfrak{l},k)$ are the Fourier coefficients of the Hilbert modular form $f_k$, $N := N_{F/\q}$ denotes the field norm, $\epsilon_K$ is the quadratic idele class character attached to $K$, $D_K$ is the discriminant of $K$ and $\mathbf{1}_\mathfrak{m}(\mathfrak{a}) = \begin{cases}
    1 & \mathfrak{a} \nmid \mathfrak{m} \\
    0 & \text{otherwise}.
\end{cases}$
\\
Then for every $\kappa \in U \cap \z_{\geq 2}$ with $\kappa \equiv 2 \hspace{1mm} mod \hspace{1mm} 2(p-1)$,
\begin{align*}
    E_{\mathfrak{l}}(f_\infty/K,\kappa) = E_{\mathfrak{l}}(f_\kappa,\mathfrak{l}^{-\kappa/2})E_{\mathfrak{l}}(f_\kappa \otimes \epsilon_K,\mathfrak{l}^{-\kappa/2})
\end{align*}
where for $\psi \in \{1,\epsilon_K\}$ and for a cusp form $g$ of parallel weight $k$, we define
\begin{align*}
    E_{\mathfrak{l}}(g,\psi,\mathfrak{l}^{-s}) := \left(1 - \frac{\psi(\mathfrak{l})C(\mathfrak{l},g)}{N(\mathfrak{l})^{-s}} + \frac{\mathbf{1}_*(\mathfrak{l})}{N(\mathfrak{l})^{2s+1-k}} \right)
\end{align*}
as the $\mathfrak{l}$-th Euler factor of the complex $L$-function of $g$ twisted by $\psi$, i.e., $L(g,\psi,s) = \prod_\mathfrak{q} E_\mathfrak{q}(g,\psi,\mathfrak{q}^{-s})^{-1}$.
\\
Define the central critical $S$-Euler factors of $f_\infty/K$ by 
\begin{align*}
    E_S(f_\infty/K,k) := \prod  E_{\mathfrak{l}}(f_\infty/K,k)
\end{align*}
where the product runs over all the primes in $S$ other than $\p$. Then for all $\mathfrak{l}|\mathfrak{n}D_K$ we have that $ E_{\mathfrak{l}}(f_\infty/K,2) \neq 0$. So upto shrinking $U$, if necessary, we can assume that $ E_{\mathfrak{l}}(f_\infty/K,k) \in \mathcal{A}(U)^\times$ for every $\mathfrak{l} \neq \p$ in $S$.
\\
Finally, define the central critical $p$-adic $L$-function of $f_\infty/K$:
\begin{align*}
    L_p^{cc}(f_\infty/K,k) := E_S(f_\infty/K,k)^{-1} \cdot L_p^S(f_\infty/K,k,k/2,1,1) \in \mathcal{A}(U).
\end{align*}

\subsection{Central Critical Selmer Group} \label{cc Sel Gp}
Fix topological generators $\gamma_+ \in \Gamma_K^+$ and $\gamma \in \Gamma$ and suppose that $\gamma_{1}$ and $\gamma_2$ generate $\Gamma_K^-$ topologically. Write $\varpi_? := \gamma_? -1$ for $? \in \{+,1,2\}$. Assume that $\chi_{cy}(\gamma_+) = \gamma$ where $\chi_{cy} : \Gamma_K^+ \cong 1+p\zp \cong \Gamma$. Let
\begin{align*}
    \Theta_K^+ : Gal(K_\infty/K) = \Gamma_K^+ \times \Gamma_K^- \twoheadrightarrow \Gamma_K^+ \xrightarrow[\chi_{cy}]{\cong} \Gamma \xrightarrow{\sqrt{\cdot}} \Gamma \rightarrow \I^\times
\end{align*}
be the character induced by $\chi_{cy}$. Extend $\Theta_K^+$ uniquely to a morphism of $\I$-algebras,
\begin{align*}
    \Theta_K^+ : \I_K \rightarrow \I.
\end{align*}
Then we can see that $\mathfrak{p}_{cc} := ker(\Theta_K^+ : \I_K \rightarrow \I) = (\varpi_{cc},\varpi_1,\varpi_2) \cdot \I_K$, where $\varpi_{cc} := [\gamma] - \gamma_+^2 \in \I_K$. Moreover, this also gives the isomorphism $\ikip/\vp_{cc} \cong \I$ when we restrict to $\Theta_K := \Theta_K^+|_{\ikip}$. 
\\
We define the (cyclotomic) $S$-primitive central critical (non-strict) Greenberg Selmer group of $\f/K$ by
\begin{align*}
    Sel^{S,cc}_{F_\infty}(\f/K) := ker\left(H^1(G_{K,S_K},\T_\f \otimes_\I \I^*) \xrightarrow{\prod_{w|\p}p_{w_*}^-\circ \hspace{0.1mm} res_w} \prod_{w|\p} H^1(I_w,\T_{\f,w}^- \otimes_\I \I^*) \right).
\end{align*}
Write $X^{S,cc}_{F_\infty}(\f/K) := \mathrm{Hom}_{\zp}( Sel^{S,cc}_{F_\infty}(\f/K),\qp/\zp).$

\subsection{Control Theorem} \label{control theorem}
Now we prove a control theorem relating the Pontryagin duals of the Selmer groups of sections \ref{Wan Sel Gp} and \ref{cc Sel Gp}.

\begin{prop} \label{ControlTheorem}
    There exists a canonical isomorphism of $\I$-modules
    \begin{align*}
        X_{K_\infty}^S(\f) \otimes_{\I_{\kinf}} \I_{\kinf}/\mathfrak{p}_{cc} \cong X_{F_\infty}^{S,cc}(\f/K).
    \end{align*}
\end{prop}

\begin{proof}
    Consider the following ideals
    \begin{align*}
        \ma_1 := (\vp_1,\vp_2) \in Spec(\I_{\kinf}), \hspace{5mm} \ma_2 := (\vp_{cc}) \in Spec(\ikip).
    \end{align*}
    Now
    \begin{align*}
         T_\f(\kinf) \otimes_{\ikinf} \ikinf^*[\ma_1] & \cong T_\f(\kinf)/\ma_1 \otimes_{\ikinf/\ma_1} \ikip^* \\ 
        & \cong T_\f(\kip) \otimes_{\ikip} \ikip^*, \\ 
        T_\f(\kinf)^-_w \otimes_{\ikinf} \ikinf^*[\ma_1] & \cong T_\f(\kinf)_w^-/\ma_1 \otimes_{\ikinf/\ma_1} \ikip^* \\ 
        & \cong T_\f(\kip)^-_w \otimes_{\ikip} \ikip^*,
    \end{align*}
    for every $w|\p$ because $\I_{\kinf}/\ma_1 \cong \I_{\kip}$ and $T_\f(\kinf)/\ma_1 \cong T_\f(\kip)$. This gives us
    \begin{align*}
        Sel^S_{\kip}(\f) \cong Sel^S_{\kinf}(\f,\ma_1)
    \end{align*}
    which, in turn, gives us the isomorphism of the Pontryagin duals :
    \begin{align*}
       X^S_{\kip}(\f) \cong X_{\kinf}^S(\f) \otimes_{\ikinf} \ikinf/\ma_1.
    \end{align*}
    Since $\Theta_K \circ \epsilon_{\kip}^{-1} = \chi_\Gamma^{-1/2}$ on $G_{K,S_K}$, we have
    \begin{align*}
        T_\f(\kip)/\ma_2 & \cong T_\f(\kip) \otimes_{\ikip} \ikip/\ma_2 \\
        & \cong T_\f(\kip) \otimes_{\ikip,\Theta_K} \I = T_\f \otimes_\I \ikip(\epsilon_{\kip}^{-1}) \otimes_{\ikip,\Theta_K} \I \cong \T_\f.
    \end{align*}
    Similarly, we also have for every $w|\p$
    \begin{align*}
        T_\f(\kip)^\pm_w/\ma_2 \cong \T_{\f,w}^\pm.
    \end{align*}
    This gives us that 
    \begin{align*}
        T_\f(\kip) \otimes_{\ikip} \ikip^*[\ma_2] & \cong T_\f(\kip)/\ma_2 \otimes_{\ikip/\ma_2} \I^* \\ 
        & \cong \T_\f \otimes_\I \I^*
    \end{align*}
    and similarly, for all $w|\p$,
    \begin{align*}
        T_\f(\kip)^\pm_w \otimes_{\ikip} \ikip^*[\ma_2] \cong \T_{\f,w}^\pm \otimes_\I \I^*
    \end{align*}
    and in particular, that
    \begin{align*}
        Sel_{\kip}^S(\f,\ma_2) \cong Sel_{F_\infty}^{S,cc}(\f/K).
    \end{align*}
    Consider the following commutative diagram

    $$
    \begin{tikzcd}[sep=1.8em, font=\small]
	0 \arrow[r] & 
    Sel_{\kip}^S(\f,\ma_2) \arrow[r] \arrow[d,"\alpha_1"] &
	  H^1(G_{K,S_K},T_\f(\kip) \otimes_{\ikip} \ikip^*[\ma_2]) \arrow[r,'] \arrow[d,"\alpha_2"] &
	\prod_{w|\p} H^1(I_w,T_\f(\kip)^-_w \otimes_{\ikip} \ikip^*[\ma_2]) \arrow[d,"\alpha_3"] 
	\\
	0 \arrow[r] &
	(Sel^S_{\kip}(\f))[\ma_2] \arrow[r] &
	(H^1(G_{K,S_K},T_\f(\kip) \otimes_{\ikip} \ikip^*))[\ma_2] \arrow[r] &
	(\prod_{w|\p} H^1(I_w,T_\f(\kip)^-_w \otimes_{\ikip} \ikip^*)[\ma_2].
    \end{tikzcd}
    $$

    We claim that if $\alpha_1$ is an isomorphism, we are done. For, we then get $Sel_{F_\infty}^{S,cc}(\f/K) \cong Sel_{\kip}^S(\f)[\ma_2]$ which gives us
    \begin{align*}
        X_{F_\infty}^{S,cc}(\f/K) & \cong X_{\kip}^S(\f) \otimes_{\ikinf} \ikinf/\ma_2 \\
        & \cong X_{\kinf}^S(\f) \otimes_{\ikinf} \ikinf/\ma_1 \otimes_{\ikinf} \ikinf/\ma_2 \\
        & \cong X_{\kinf}^S(\f) \otimes_{\ikinf} \ikinf/\mathfrak{p}^{cc}.
    \end{align*}
    The second isomorphism follows from \cite[Corollary 17]{Wan}. So now it only remains to show that $\alpha_1$ is an isomorphism. We have an exact sequence
    \begin{align*}
        0 \rightarrow \ikip^*[\ma_2] \rightarrow \ikip^* \xrightarrow{\varpi_{cc}} \ikip^* \rightarrow 0
    \end{align*}
    which induces a short exact sequence
    \begin{align*}
        0 \rightarrow H^0(G_{K,S_K},T_\f(\kip) \otimes_{\ikip} \ikip^*)/\varpi_{cc} & \rightarrow H^1(G_{K,S_K},T_\f(\kip) \otimes_{\ikip} \ikip^*[\ma_2]) \\
        & \xrightarrow{\alpha_2} H^1(G_{K,S_K},T_\f(\kip) \otimes_{\ikip} \ikip^*)[\ma_2] \rightarrow 0.
    \end{align*}
    The hypothesis \textbf{(irred)} implies that $\overline{\rho_\f}|_{G_K}$ is irreducible. So $H^0(G_{K,S_K},T_\f(\kip) \otimes_{\ikip} \ikip^*) = 0$. Hence, $\alpha_2$ is an isomorphism. So in order to prove that $\alpha_1$ is an isomorphism it is enough to show that $ker(\alpha_3) = 0$. As earlier, we have a short exact sequence
    \begin{align*}
        0 \rightarrow \prod_{w|\p} H^0(I_w,T_\f(\kip)^-_w \otimes_{\ikip} \ikip^*)/\varpi_{cc} & \rightarrow \prod_{w|\p} H^1(I_w,T_\f(\kip)^-_w \otimes_{\ikip} \ikip^*[\ma_2]) \\
        & \xrightarrow{\alpha_3} \left(\prod_{w|\p} H^1(I_w,T_\f(\kip)^-_w \otimes_{\ikip} \ikip^*) \right)[\ma_2] \rightarrow 0
    \end{align*}
    using which we get that $ker(\alpha_3) = \prod_{w|\p} H^0(I_w,T_\f(\kip)^-_w \otimes_{\ikip} \ikip^*) \otimes_{\ikip} \ikip/\varpi_{cc}$. We show that each module in the product in $ker(\alpha_3)$ is trivial. For this it suffices to show that each $H^0(I_w,T_\f(\kip)^-_w \otimes_{\ikip} \ikip^*)$ is $\vp_{cc}$-divisible.  Since
    \begin{align*}
        T_\f(\kip)^-_w \otimes_{\ikip} \ikip^* & = T_{\f,w}^- \otimes_\I \ikip(\epsilon_{\kip}^{-1}) \otimes_{\ikip} \ikip^* \\
        & \cong \ikip^*(\ma_\p^*\epsilon_{\kip}^{-1}),
    \end{align*}
    $H^0(I_w,T_\f(\kip)^-_w \otimes_{\ikip} \ikip^*) = \ikip^*(\ma_\p^*)[\gamma_+ - 1] = \I^*(\ma_\p^*)$. \\
    Now $\varpi_{cc} = [\gamma] - \gamma^2_+$ acts as $[\gamma] - 1$ on $\I^*$. So $\I^*$ and hence each $H^0(I_w,T_\f(\kip)^-_w \otimes_{\ikip} \ikip^*)$ is $\varpi_{cc}$-divisible. As a result we get that $ker(\alpha_3) = 0$.
\end{proof}

\begin{cor} \label{ineq2}
    Assume \textbf{(irred)} and the hypotheses in Theorem \ref{WanMainResult}. Then
    \begin{align*}
        \mathrm{ord}_{k=2} L_p^{cc}(f_\infty/K,k) \leq \mathrm{len}_{P}(X_{F_\infty}^{S,cc}(\f/K)).
    \end{align*}
\end{cor}

\begin{proof}
     Write $\mcL_{F_\infty}^{S,cc}(\f/K) := \mcL_K^S(\f) \hspace{1mm} mod \hspace{1mm} \mathfrak{p}_{cc}$. First, we show that 
     \begin{align*}
         \mathrm{ord}_P(\mcL_{F_\infty}^{S,cc}(\f/K)) = \mathrm{ord}_{k=2} L_p^S(f_\infty/K,k,k/2,1,1). 
     \end{align*}
      By the definition of $\widetilde{M}$ we have
    \begin{align*}
        \widetilde{M}(\varpi_{cc})(k,s,r,t) = \gamma^{k-2} - \gamma^{2(s-1)} & = \gamma^{2(s-1)} (\gamma^{2(k/2-s)}-1) \\
        & \equiv 0 \hspace{1mm} mod \hspace{1mm} (s-k/2) \hspace{1mm} \mathcal{A}(U \times \zp \times \zp \times \zp).
    \end{align*}
    Similarly, writing $l_{wt} := \log_p(\gamma), l_1 := \log_p(\chi_{acy}(\gamma_1))$ and $l_2 := \log_p(\chi_{acy}(\gamma_2))$, we have
    \begin{align} \label{value of M at varpi wt}
        M(\varpi)(k) & \equiv l_{wt}(k-2) \hspace{1mm} mod \hspace{1mm} (k-2)^2,
    \end{align}
    \begin{align*} 
        \widetilde{M}(\varpi_1)(k,s,r,t) & \equiv l_1(r-1) \hspace{1mm} mod \hspace{1mm} (r-1)^2 \\
        \widetilde{M}(\varpi_2)(k,s,r,t) & \equiv l_2(t-1) \hspace{1mm} mod \hspace{1mm} (t-1)^2. 
    \end{align*}
    In particular, 
    \begin{align} \label{value of M at varpi cc}
         \widetilde{M}(\varpi_{cc})(k,k/2,1,1) & = 0,  
    \end{align}
    \begin{align} \label{value of M at varpi 1}
        \widetilde{M}(\varpi_1)(k,k/2,1,1) & = 0,
    \end{align}
    \begin{align} \label{value of M at varpi 2}
        \widetilde{M}(\varpi_2)(k,k/2,1,1) & = 0.
    \end{align}
    Now assume that for some integer $m \geq 0, \mathrm{ord}_P(\mcL_{F_\infty}^{S,cc}(\f/K)) = m$, i.e., $\mcL_{F_\infty}^{S,cc}(\f/K) \in P^m\I_P - P^{m+1}\I_P$. Since $P\I_P = \varpi \I_P$, then $\mcL_{F_\infty}^{S,cc}(\f/K) \in (\varpi^m)$ and so using (\ref{value of M at varpi wt}),
    \begin{align*}
        \mathrm{ord}_{k=2}M(\mcL_{F_\infty}^{S,cc}(\f/K))(k) = m = \mathrm{ord}_P(\mcL_{F_\infty}^{S,cc}(\f/K)).
    \end{align*}
    Since $\mcL_K^S(\f) - \mcL_{F_\infty}^{S,cc}(\f/K) \in \mathfrak{p}_{cc} = (\varpi_{cc},\varpi_1,\varpi_2)$, we get that
    \begin{align*}
        L_p^S(f_\infty/K,k,k/2,1,1) = M(\mcL_{F_\infty}^{S,cc}(\f/K))(k).
    \end{align*}
    So in this case, we get the equality. Now assume that $\mathrm{ord}_P(\mcL_{F_\infty}^{S,cc}(\f/K)) = \infty$, i.e., $\mcL_K^S(\f) \in \mathfrak{p}_{cc} = (\varpi_{cc},\varpi_1,\varpi_2)$. Then using (\ref{value of M at varpi cc}),(\ref{value of M at varpi 1}) and (\ref{value of M at varpi 2}), we get that 
    \begin{align*}
        L_p^S(f_\infty/K,k,k/2,1,1) = 0.
    \end{align*}
    So in both cases, we get that $\mathrm{ord}_P(\mcL_{F_\infty}^{S,cc}(\f/K)) = \mathrm{ord}_{k=2} L_p^S(f_\infty/K,k,k/2,1,1)$. Now by Proposition \ref{ControlTheorem}, \cite[Corollary 3.8]{SkinnerUrban} and \cite[Theorem 101]{Wan}, we have
    \begin{align*}
        \mathrm{char}(X_{F_\infty}^{S,cc}(\f/K)) = \mathrm{char}(X_{K_\infty}^S(\f)/\mathfrak{p}_{cc}) \subset \mathrm{char}(X_{K_\infty}^S(\f)) \hspace{1mm} mod \hspace{1mm} \mathfrak{p}_{cc} \subset (\mcL_K^S(\f)) \hspace{1mm} mod \hspace{1mm} \mathfrak{p}_{cc}.
    \end{align*}
    In particular,
    \begin{align*}
        \mathrm{ord}_P(\mcL_{F_\infty}^{S,cc}(\f/K)) \leq \mathrm{len}_P(X_{F_\infty}^{S,cc}(\f/K))
    \end{align*}
    which gives us the desired result.
\end{proof}

\section{Vanishing of the Central Critical $p$-adic $L$-Function}
In this section, we discuss the order of vanishing of the central critical $p$-adic $L$-function $L_p^{cc}(f_\infty/K,k)$ using the main result of \cite{MokHeegnerPoints}. \\
Consider the two-variable $p$-adic $L$-function $L_p(f_\infty,\psi,k,k/2)$ defined in \cite[Section 8]{MokExceptionalZero} for $\psi \in \{1,\epsilon_K\}$.
\\
Now, let $\kappa \in U \cap \z_{\geq 2}$ such that $\kappa \equiv 2 \hspace{1mm} mod \hspace{1mm} 2(p-1)$. Let $\phi_\kappa$ be the associated arithmetic point in $\mathfrak{X}_{alg}(\I)$ and let $f_\kappa$ be the associated newform. Write
\begin{align*}
    \phi_\kappa^\dagger : \I[[\Gamma_K^+ \times \Gamma_K^-]] \rightarrow \overline{\qp}
\end{align*}
for the morphism for which
\begin{align*}
    \phi_\kappa^\dagger|_\I = \phi_\kappa, \hspace{5mm} \phi_\kappa^\dagger(\Gamma_K^-) = 1, \hspace{5mm} \phi_\kappa^\dagger(\sigma) = \chi_{cy}(\sigma)^{\kappa/2 - 1}
\end{align*}
for every $\sigma \in \Gamma_K^+$. Note that $\phi_\kappa^\dagger = \widetilde{M}$ on $(k,k/2,1,1)$. Then using the interpolation property given in \cite[Theorem 1.1.2]{BergdallHansen} and (\ref{FactorisationInLZ}), we get that
\begin{align*}
    \phi_\kappa^\dagger(\mathcal{L}^S_K(\f)) & = \chi_{cy}^{\kappa/2 - 1} \circ \phi_\kappa^{cy}(\mathcal{L}^S_K(\f)) \\
    & = \chi_{cy}^{\kappa/2 - 1}(c \cdot L_p^S(f_\kappa) \cdot L_p^S(f_\kappa \otimes \epsilon_K)) \\
    & = \lambda \left(1 - \frac{N\p^{\kappa/2 - 1}}{\alpha(\p,f_\kappa)} \right)^2 \frac{D_F^{\kappa/2 - 1}(\kappa/2 - 1)!^2}{(-2\pi i)^{2(\kappa/2-1)}} \frac{L^S(f_\kappa,\kappa/2)}{\Omega_{f_{\kappa}}^{(-1)^{\kappa/2-1}}} \\
    & \times \left(1 - \frac{N\p^{\kappa/2 - 1}}{\alpha(\p,f_\kappa \otimes \epsilon_K)} \right)^2 \frac{D_F^{\kappa/2 - 1}(\kappa/2 - 1)!^2 Nc_K^{\kappa/2}}{(-2\pi i)^{2(\kappa/2-1)} \tau(\epsilon_K^{-1})} \frac{L^S(f_\kappa \otimes \epsilon_K,\kappa/2)}{\Omega_{f_{\kappa} \otimes \epsilon_K}^{(-1)^{(\kappa/2-1)}sig(\epsilon_K)}}
\end{align*}
where $\lambda \in \overline{\qp}^*$, $D_F$ is the discriminant of the field $F$, $L^S(h,-)$ is the complex $L$-function of the Hilbert modular form $h$ with Euler factors at points in $S$ removed, $\Omega_{h}$ is a choice of period for $h$, $c_K$ is the conductor of the quadratic idele class character $\epsilon_K$ attached to the field $K$, $\tau(\epsilon_K^{-1})$ is the Gauss sum of $\epsilon_K^{-1}$ and $sig(\epsilon_K)$ is the signature of $\epsilon_K$ as defined in \cite[Section 5.1]{MokHeegnerPoints}. \\
Now for every $\kappa \in U \cap \z_{\geq 2}$ such that $\kappa \equiv 2 \hspace{1mm} mod \hspace{1mm} 2(p-1)$, recall we had
\begin{align*}
    L_p^{cc}(f_\infty/K,\kappa) & = \prod_{\mathfrak{l} \in S} E_\mathfrak{l}(f_\kappa, \mathfrak{l}^{-\kappa/2})^{-1} E_\mathfrak{l}(f_\kappa \otimes \epsilon_K, \mathfrak{l}^{-\kappa/2})^{-1} \widetilde{M}(\mathcal{L}^S_K(\f)) \\
    & = \prod_{\mathfrak{l} \in S} E_\mathfrak{l}(f_\kappa, \mathfrak{l}^{-\kappa/2})^{-1} E_\mathfrak{l}(f_\kappa \otimes \epsilon_K, \mathfrak{l}^{-\kappa/2})^{-1} \phi_\kappa^\dagger(\mathcal{L}^S_K(\f)) \\
    & = \prod_{\mathfrak{l} \in S} E_\mathfrak{l}(f_\kappa, \mathfrak{l}^{-\kappa/2})^{-1} E_\mathfrak{l}(f_\kappa \otimes \epsilon_K, \mathfrak{l}^{-\kappa/2})^{-1} \\
    & \lambda \left(1 - \frac{N\p^{\kappa/2 - 1}}{\alpha(\p,f_\kappa)} \right)^2 \frac{D_F^{\kappa/2 - 1}(\kappa/2 - 1)!^2}{(-2\pi i)^{2(\kappa/2-1)}} \frac{L^S(f_\kappa,\kappa/2)}{\Omega_{f_{\kappa}}^{(-1)^{\kappa/2-1}}} \\
    & \times \left(1 - \frac{N\p^{\kappa/2 - 1}}{\alpha(\p,f_\kappa \otimes \epsilon_K)} \right)^2 \frac{D_F^{\kappa/2 - 1}(\kappa/2 - 1)!^2 Nc_K^{\kappa/2}}{(-2\pi i)^{2(\kappa/2-1)} \tau(\epsilon_K^{-1})} \frac{L^S(f_\kappa \otimes \epsilon_K,\kappa/2)}{\Omega_{f_{\kappa} \otimes \epsilon_K}^{(-1)^{(\kappa/2-1)}sig(\epsilon_K)}} \\
    & = \lambda \left(1 - \frac{N\p^{\kappa/2 - 1}}{\alpha(\p,f_\kappa)} \right)^2 \frac{D_F^{\kappa/2 - 1}(\kappa/2 - 1)!^2}{(-2\pi i)^{2(\kappa/2-1)}} \frac{L(f_\kappa,\kappa/2)}{\Omega_{f_{\kappa}}^{(-1)^{\kappa/2-1}}} \\
    & \times \left(1 - \frac{N\p^{\kappa/2 - 1}}{\alpha(\p,f_\kappa \otimes \epsilon_K)} \right)^2 \frac{D_F^{\kappa/2 - 1}(\kappa/2 - 1)!^2 Nc_K^{\kappa/2}}{(-2\pi i)^{2(\kappa/2-1)} \tau(\epsilon_K^{-1})} \frac{L(f_\kappa \otimes \epsilon_K,\kappa/2)}{\Omega_{f_{\kappa} \otimes \epsilon_K}^{(-1)^{(\kappa/2-1)}sig(\epsilon_K)}}.
\end{align*}
So by the interpolation property given in \cite{MokHeegnerPoints} and using the fact that $\{\kappa \in U \cap \z_{\geq 2} | \kappa \equiv 2 \hspace{1mm} mod \hspace{1mm} 2(p-1) \}$ is a dense subset of $U$, we get that for every $k \in U$,
\begin{align} \label{Central Critical in Terms of Mok}
    L_p^{cc}(f_\infty/K,k) = \lambda L_p(f_\infty,1,k,k/2) L_p(f_\infty,\epsilon_K,k,k/2).
\end{align}
The following is \cite[Theorem 5.4]{MokHeegnerPoints}.
\begin{thm}  \label{MokMainResult}
    Let $\psi$ be a quadratic idele class character of $F$, of conductor prime to $\n_E = \mathfrak{np}$. Assume that $\psi(\p) = 1$ and $sgn(\psi)\psi(\mathfrak{n}) = -w(E/F)$, where $sgn(\psi) := \prod_{w|\infty} \psi_w(-1)$ and $w(E/F) \in \{\pm 1\}$ is the global root number of $E/F$, i.e., the sign appearing in the functional equation of Hasse-Weil $L$-function of the elliptic curve $E$ over $F$. Let $K'$ be the field attached to $\psi$. Then
    \begin{enumerate}
        \item The function $L_p(f_\infty,\psi,k,k/2)$ vanishes to order atleast 2 at $k = 2$.
        \item There exists an element $\mathbb{P}_\psi \in (E(K') \otimes \q)_\psi$ and $l \in \q^*$ such that
        \begin{align*}
            \frac{d^2}{dk^2} L_p(f_\infty,\psi,k,k/2) \vline_{k=2} = l \cdot (\log\mathrm{Norm}_{E/F_\p}(\mathbb{P}_\psi))^2,
        \end{align*}
        where $\log\mathrm{Norm}_{E/F_\p}$ is the map defined in \cite[Section 4.5]{MokHeegnerPoints}.
        \item $\mathbb{P}_\psi$ is of infinite order if and only if $L'(E/F,\psi,1) \neq 0$.
    \end{enumerate}
\end{thm}

\begin{cor} \label{ineq1}
    Assume that $w(E/F) = -1$ and that the hypotheses in Theorem \ref{WanMainResult} are satisfied. Let $\epsilon_K(\p) = 1$. Then $L_p^{cc}(f_\infty/K,k)$ vanishes to order at least 4 at $k = 2$. Further, $\mathrm{ord}_{k=2} L_p^{cc}(f_\infty/K,k) = 4$ if and only if $\mathrm{ord}_{s=1}L(E/K,s) = 2$.
\end{cor}
\begin{proof}
    First, let $\psi$ be the trivial idele class character. Since $w(E/F) = -1$, clearly the hypothesis of Theorem \ref{MokMainResult} are satisfied by $\psi$. Now let $\psi = \epsilon_K$. Then by our hypothesis $\epsilon_K(\p) = 1$ and since $F$ only has real embeddings, $sgn(\epsilon_K) = 1$. So again by our hypothesis we get $\psi(\mathfrak{n}) = \psi(\mathfrak{n}^+ \mathfrak{n}^-) = \psi(\mathfrak{n}^-) = 1$. Now note that $L(E/K,s) = L(E/F,s)L(E/F,\epsilon_K,s)$ and apply Theorem \ref{MokMainResult} to both the trivial character and the quadratic character $\epsilon_K$ to obtain the result.
\end{proof}

\section{Nekov\'ar's Selmer Complexes and the Extended Mordell-Weil Group} \label{sec 6}
In this section, we recall Selmer complexes as defined in \cite{Nekovar}. Further, we introduce the extended Mordell-Weil group, and relate it with the first cohomology group of the Selmer complex (see (\ref{defn of iedagger})).
\\
Let $\chi \in \{1,\epsilon_K\}$ be a quadratic character of $F$ of conductor coprime with $\mathfrak{np}$. If $\chi$ is trivial, set $K'=F$, otherwise let $K' = K$. Also, set $S' \in \{S,S_K\}$ accordingly. For all $w \in S'$, fix an embedding $i_w : \overline{K'} \hookrightarrow \overline{K'_w}$. This induces a morphism $i_w^* : G_w \hookrightarrow G_{K'}$. 

\subsection{Kummer Theory} \label{kummer}
We begin by recalling Kummer theory.
\\
Recall that $T_p(E) := \varprojlim_n E(\overline{\q})[p^n]$ is the Tate module of $E/F$. As $S'$ contains all the bad primes, $T_p(E)$ is a continuous $G_{K',S'}$-module. For every $w \in S'$, $i_w$ induces an isomorphism of $G_w$-modules, (also denoted $i_w$), $i_w : (E(\overline{\q})[p^n])_w \xrightarrow{\sim} E(\overline{K'_w})[p^n]$. (Here, for any $\z[G_{K',S'}]$-module $M$, we denote by $M_w$ the $\z[G_w]$-module $M$ with the $G_w$-action induced by $i_w^*$). Taking the inverse limit over $n$ and base changing to $\qp$ gives an isomorphism of $\qp[G_w]$-modules
\begin{align*}
    i_w : V_p(E)_w \xrightarrow{\sim} V_w(E) 
\end{align*}
where $V_w(E) := \varprojlim E(\overline{K'_w})[p^n] \otimes \qp$. Consider the global Kummer map
\begin{align*}
   \kappa : E(K') \otimes \qp = \varprojlim E(K')/p^nE(K') \otimes \qp \xrightarrow{\kappa_n} \varprojlim_n H^1(G_{K',S'},E[p^n]) \otimes \qp \xrightarrow{\sim} H^1(G_{K',S'},V_p(E))
\end{align*}
where $\kappa_n$ is the usual Kummer map. Denote by $\gamma_x$ the image of $x \otimes 1 \in E(K') \otimes \qp$ under this map. Similarly define the local Kummer map for every $w \in S'$
\begin{align*}
    \kappa_w : E(K'_w) \otimes \qp = \varprojlim E(K'_w)/p^nE(K'_w) \otimes \qp \xrightarrow{\kappa_{w,n}} \varprojlim_n H^1(K'_w,E[p^n]) \otimes \qp \xrightarrow{\sim} H^1(K'_w,V_w(E))
\end{align*}
and the element $\gamma_{x_w} \in H^1(K'_w,V_w(E))$ for every $x_w \in E(K'_w)$. The isomorphism $i_w$ induces a map in cohomology $i_w^{-1} : H^1(K'_w,V_w(E)) \rightarrow H^1(K'_w,V_p(E))$. We also have the restriction map $res_w : H^1(G_{K',S'},V_p(E)) \rightarrow H^1(K'_w,V_p(E))$. Then for every $w \in S'$, the following diagram commutes
$$ 
\begin{tikzcd} 
	E(K') \otimes \qp \arrow[r,"\kappa"] \arrow[d,"i_w \otimes id"] &
	H^1(G_{K',S'},V_p(E)) \arrow[d,"res_w"] 
	\\
	E(K'_w) \otimes \qp \arrow[r,"i_w^{-1} \circ \kappa_w"] &
	H^1(K'_w,V_p(E))
\end{tikzcd}
$$
so that we have, 
\begin{align} \label{comm diag kummer theory}
    res_w(\gamma_x) = i_w^{-1}(\gamma_{i_w(x)}). 
\end{align}
\\
The local conditions $E(K'_w) \otimes \qp$ define the Selmer group $Sel_E(K') \subset H^1(G_{K',S'},V_p(E))$ which fits in the following exact sequence
\begin{align} \label{ses for MW group}
    0 \rightarrow E(K') \otimes L \rightarrow Sel_E(K') \otimes L \rightarrow T_p(\Varpi(E/K')) \otimes L \rightarrow 0,
\end{align}
where $T_p(\Varpi(E/K')) := \varprojlim_{n \geq 1} \Varpi(E/K')[p^n]$ and $L$ is a finite extension of $\qp$ as in Section \ref{BigGaloisRepresentation}.
\\
We give another description of the Selmer group.
\\
Since $E/F_\mathfrak{p}$ has split multiplicative reduction, Tate's parametrization gives us a group isomorphism,
\begin{align*}
    \Phi_{Tate} : \overline{\qp}^*/q_E^\z \xrightarrow{\sim} E(\overline{\qp})
 \end{align*}
where $q_E \in F_\mathfrak{p}$ is the Tate period and $\mathrm{ord}_p(q_E)>0$. Restricting to $p^n$-torsion points for every $n \in \z_{>0}$, this isomorphism induces a short exact sequence
\begin{align*} 
    0 \rightarrow \mu_{p^n} \xrightarrow{\Phi_{Tate}} E(\overline{\qp})_{p^n} \xrightarrow{P_{Tate}} \z/p^n\z \rightarrow 0.
\end{align*}
Taking the inverse limit over $n$ and base changing to $L$ (a finite extension of $\qp$ as in Section \ref{BigGaloisRepresentation}) we get
\begin{align} \label{ses for Vf}
    0 \rightarrow L(1) \xrightarrow{p_w^+} V_f \xrightarrow{p_w^-} L \rightarrow 0.
\end{align}
Here $p_w^+$ is induced by $i_w^{-1} \circ \Phi_{Tate}$ and $p_w^-$ is induced by $P_{Tate} \circ i_w$. Moreover, if $\Phi_{Tate}(\widetilde{P}) = P \in E(K'_w)$ for some $\widetilde{P} \in (K'_w)^*$ then we have that $(\Phi_{Tate})_*(\gamma_{\widetilde{P}}) = \gamma_P$, where $(\Phi_{Tate})_*$ is the map induced by $\Phi_{Tate} : L(1) \rightarrow V_f$ in cohomology. For every $w|\mathfrak{p}$ in $S'$, define $H^1_f(K'_w,V_f) \subset H^1(K'_w,V_f)$ as the image of $H^1(K'_w,L(1))$ under the map induced in cohomology by $p_w^+$ and for every $w \nmid \mathfrak{p}$ in $S'$, define $H^1_f(K'_w,V_f)=0$. These local conditions define a Selmer group $H^1_f(K',V_f) \subset H^1(K',V_f)$. Then as in \cite[Section 2]{greenberg1998iwasawatheoryellipticcurves}, we get
\begin{lem}
    $Sel_E(K') \otimes_{\qp} L = H^1_f(K',V_f)$.
\end{lem}

\subsection{Selmer Complexes}
Now we give some preliminaries on Selmer complexes following \cite{Nekovar}.
\\
Write
\begin{align*}
    \mathds{A}_\f & := \mathrm{hom}_{cont}(\T_\f,\mu_{p^\infty}), \hspace{5mm} \mathds{A}_\f^\pm := \mathrm{hom}_{cont}(\T_\f^\mp,\mu_{p^\infty})
\end{align*}
and $\T_{\f,P}$ (resp., $\T_{\f,P}^\pm$) for the localisation of $\T_\f$ (resp., $\T_\f^\pm$) at the prime $P$.
Recall from Section \ref{BigGaloisRepresentation} we have the following short exact sequences
\begin{align*}
    0 \rightarrow V_f^+ \rightarrow V_f \rightarrow V_f^- \rightarrow 0, \\
    0 \rightarrow \T_\f^+ \rightarrow \T_\f \rightarrow \T_\f^- \rightarrow 0,
\end{align*}
where $V_f^+ \cong L(1)$ and $V_f^- \cong L$. The following short exact sequences are induced by the exact sequence above:
\begin{align*}
    0 \rightarrow \T_{\f,P}^+ \rightarrow \T_{\f,P} \rightarrow \T_{\f,P}^- \rightarrow 0, \\
    0 \rightarrow \mathds{A}_\f^+ \rightarrow \mathds{A}_\f \rightarrow \mathds{A}_\f^- \rightarrow 0.
\end{align*}
Let $w \in S'$ be a prime of $K'$ lying above a prime $v$ of $F$. For $X \in \{V_f,\T_\f,\T_{\f,P}, \mathds{A}_\f \}$, put $X_v^\pm := X^\pm$ and then set the local conditions as 
\begin{align*}
    U_w^+(X) := \begin{cases}
        C^\bullet_{cont}(K'_w,X_v^+) & v|p \\
        0 & v \nmid p
    \end{cases}
\end{align*}
and set $i_w^+ : U_w^+(X) \rightarrow C^\bullet_{cont}(K'_w,X)$ for the map induced by the inclusion $X_v^+ \hookrightarrow X$. Following \cite[Section 6.1]{Nekovar} we define the Selmer complex attached to $X$ as
\begin{align*}
    \widetilde{C}^\bullet_f(G_{K',S'},X) = \text{Cone} \left(C^\bullet_{cont}(G_{K',S'},X) \oplus \bigoplus_{w \in S'} U_w^+(X) \xrightarrow{res_{S'} - i_{S'}^+} \bigoplus_{w \in S'} C^\bullet_{cont}(K'_w,X)\right)[-1],
\end{align*}
where $res_{S'} := \oplus_{w \in S'} res_w$ and $i_{S'}^+ := \oplus_{w \in S'} i_w^+$. Write $\widetilde{\textbf{R}\Gamma}_f(G_{K',S'},X)$ for the image of $\widetilde{C}^\bullet_f(G_{K',S'},X)$ in the derived category of $R$-modules where $R = \I$ if $X \in \{\T_\f, \mathbb{A}_\f\}$, $R = \mathcal{O}_{L,P}$ if $X = V_f$ and $R = \ip$ if $X = \tfp$. Also write $\widetilde{H}^*_f(G_{K',S'},X)$ for the cohomology groups $H^*(\widetilde{C}^\bullet_f(G_{K',S'},X))$. Recall that $S'$ contains all the bad primes.

\begin{prop} \label{SesForSelmerComplex}
    \begin{enumerate}
        \item Upto a canonical isomorphism, the Selmer complex $\widetilde{\textbf{R}\Gamma}_f(G_{K',S'},X)$ does not depend on the choice of $S'$ and hence we denote it by $\widetilde{\textbf{R}\Gamma}_f(K',X)$ and cohomology groups by $\widetilde{H}^i_f(K',X)$.

        \item There is an exact triangle
        \begin{align*}
             \widetilde{\textbf{R}\Gamma_f}(K',\T_{\f,P}) \xrightarrow{\varpi} \widetilde{\textbf{R}\Gamma_f}(K',\T_{\f,P}) \rightarrow\widetilde{\textbf{R}\Gamma_f}(K',V_f)
        \end{align*}
        inducing a short exact sequence
        \begin{align*}
            0 \rightarrow \widetilde{H}^q_f(K',\T_{\f,P})/(\varpi) \rightarrow \widetilde{H}^q_f(K',V_f) \xrightarrow{i_P} \widetilde{H}^{q+1}_f(K',\T_{\f,P})[\varpi] \rightarrow 0.
        \end{align*}

        \item  $\widetilde{H}^1_f(K',\T_{\f,P})$ is a free $\ip$-module.

        \item There is an exact sequence 
        \begin{align}
            0 \rightarrow \bigoplus_{w|\p} L \rightarrow \widetilde{H}^1_f(K',V_f) \rightarrow H^1_f(K',V_f) \rightarrow 0.
        \end{align}
    \end{enumerate}
\end{prop}

\begin{proof}
    See propositions 12.7.13.3 and 12.7.13.4 of \cite{Nekovar}.
\end{proof}

\subsubsection{Global Cup-Product Pairing}
Recall we had a perfect bilinear pairing 
\begin{align*}
    \pi : \T_\f \otimes \T_\f \rightarrow \I(1)
\end{align*}
in Section \ref{BigGaloisRepresentation} which gives rise to the pairing
\begin{align*}
    \pi_P : \tfp \otimes \tfp \rightarrow \ip(1).
\end{align*} 
\cite[Section 6.3]{Nekovar} attached to $\pi_P$ the following global cup-product pairing :
\begin{align*}
    \cup_{Nek} : \widetilde{H}^q_f(K',\T_{\f,P}) \otimes_{\ip} \widetilde{H}^{3-q}_f(K',\T_{\f,P}) \rightarrow \ip
\end{align*}
which gives rise to the following adjunction isomorphisms (see \cite[Proposition 6.7.7]{Nekovar})
\begin{align} \label{H^1-H^2 relation}
    adj(\cup_{Nek}) : \widetilde{H}^q_f(K',\T_{\f,P}) \otimes_{\ip} Frac(\ip) \cong \mathrm{hom}_{Frac(\ip)}(\widetilde{H}^{3-q}_f(K',\T_{\f,P}) \otimes_{\ip} Frac(\ip), Frac(\ip)).
\end{align}

\subsubsection{Generalised Cassels-Tate Pairing} \label{CT Pairing}
Following \cite[Section 10.2]{Nekovar} we define generalised Cassels-Tate cup-product pairing. Let 
\begin{align*}
    \mathfrak{C} := Cone \left(\tau_{\geq 2}C^\bullet_{cont}(G_{K',S'},\I(1)) \xrightarrow{res_{S'}} \oplus_{w \in S'} \tau_{\geq 2} C^\bullet_{cont}(K'_w,\I(1))\right)[-1]. 
\end{align*}
For $G \in \{G_{K',S'},G_w\}$, this induces the usual cup product pairing
\begin{align*}
    \cup_\pi : C^\bullet_{cont}(G,\tfp) \otimes_{\ip} C^\bullet_{cont}(G,\tfp) \rightarrow C^\bullet_{cont}(G,\ip(1)).
\end{align*}
By \cite[Proposition 1.3.2]{Nekovar}, for any $r \in \I$ we have a morphism
\begin{align*}
    \cup_{\pi,r} : \widetilde{C}^\bullet_f(G_{K',S'},\tfp) \otimes_{\ip} \widetilde{C}^\bullet_f(G_{K',S'},\tfp) \rightarrow \mathfrak{C}
\end{align*}
defined as
\begin{align*}
    (x_n,x_n^+,x_{n-1}) \cup_{\pi,r} (y_m,y_m^+,y_{m-1}) := (x_n \cup_\pi y_m, \left( x_{n-1} \cup_\pi(r \cdot res_{S'}(y_m) + (1-r) \cdot i_{S'}^+(y_m^+)) \right) + \\
    \left( (-1)^n((1-r) \cdot res_{S'}(x_n) + r \cdot i_{S'}^+(x_n^+)) \cup_\pi y_{m-1}) \right).
\end{align*}
This morphism is independent of $r$ upto homotopy, so we omit $r$ in the notation and write $\cup_\pi$ for $\cup_{\pi,r}$. Let $\mathscr{I} := Frac(\ip)$ and $\mathcal{C} := [\ip \xrightarrow{-i} \mathscr{I}]$ be the complex that is concentrated in degrees 0 and 1. By definition
\begin{align*}
    \mathcal{C} \otimes_{\ip} \mathcal{C} = [\ip \xrightarrow{(-i,-i)} \mathscr{I} \oplus \mathscr{I} \xrightarrow{(-id,id)} \mathscr{I}]
\end{align*}
is the complex concentrated in degrees 0, 1 and 2. It is easy to see that the morphism
\begin{align*}
    v_{\ip} : \mathcal{C} \otimes \mathcal{C} \rightarrow \mathcal{C}
\end{align*}
defined by the identity in degree 0 and by projection on the first component in degree 1 is a quasi-isomorphism. In order to ease the notation, write $\widetilde{C}^\bullet_f(\tfp)$ for $\widetilde{C}^\bullet_f(G_{K',S'},\tfp)$. Define a morphism
\begin{align*}
    \left(\widetilde{C}^\bullet_f(\tfp) \otimes_{\ip} \mathcal{C} \right) \otimes_{\ip} \left(\widetilde{C}^\bullet_f(\tfp) \otimes_{\ip} \mathcal{C} \right) \xrightarrow{s_{23}} \left(\widetilde{C}^\bullet_f(\tfp) \otimes_{\ip} \widetilde{C}^\bullet_f(\tfp) \right) \otimes_{\ip} \left(\mathcal{C} \otimes_{\ip} \mathcal{C} \right) \\ \xrightarrow{id \otimes v_{\ip}}\left(\widetilde{C}^\bullet_f(\tfp) \otimes_{\ip} \widetilde{C}^\bullet_f(\tfp) \right) \otimes_{\ip} \mathcal{C} \xrightarrow{\cup_\pi \otimes id} \mathfrak{C} \otimes_{\ip} \mathcal{C}
\end{align*}
where $s_{23}((a \otimes b) \otimes (c \otimes d)) := (-1)^{deg(b)deg(c)}((a \otimes c) \otimes (b \otimes d))$ and $\cup_\pi := \cup_{\pi,r}$ for some $r \in \I$. This induces in cohomology a morphism which is independent of the choice of $r$,
\begin{align*}
    H^2(\widetilde{C}^\bullet_f(\tfp) \otimes_{\ip} \mathcal{C}) \otimes_{\ip} H^2(\widetilde{C}^\bullet_f(\tfp) \otimes_{\ip} \mathcal{C}) \xrightarrow{\cup_{\pi,2,2}} H^4(\mathfrak{C} \otimes_{\ip} \mathcal{C}). 
\end{align*}
$\cup_{\pi,2,2}$ factorises through the projection $H^2(\widetilde{C}^\bullet_f(\tfp) \otimes_{\ip} \mathcal{C}) \rightarrow H^2(\widetilde{C}^\bullet_f(\tfp))_{\ip-tor}$. The exact triangle
\begin{align*}
    \mathfrak{C} \rightarrow \mathfrak{C} \otimes_{\ip} \mathscr{I} \rightarrow (\mathfrak{C} \otimes_{\ip} \mathcal{C})[1]
\end{align*}
induces a short exact sequence
    \begin{align} \label{ses}
        0 \rightarrow H^{q-1}(\mathfrak{C}) \otimes_{\I_P} \mathscr{I}/\I_P \rightarrow H^q(\mathfrak{C} \otimes_{\I_P} \mathcal{C}) \rightarrow H^q(\mathfrak{C})_{\I_P-tor} \rightarrow 0.
    \end{align}
Moreover, we have that $H^4(\mathfrak{C}) = 0$ (see \cite[Section 5.4.1]{Nekovar}). So, in particular for $q = 4$, we obtain the following isomorphims
\begin{align*}
    H^4(\mathfrak{C} \otimes_{\ip} \mathcal{C}) \xrightarrow{\sim} H^3(\mathfrak{C}) \otimes_{\ip} \mathscr{I}/\ip \xrightarrow{\sim} \mathscr{I}/\ip.
\end{align*}
So we get an $\ip$-bilinear form
\begin{align} \label{Cassels Tate Pairing}
    \cup_{CT} : \widetilde{H}^2_f(K',\tfp)_{\ip-tor} \times \widetilde{H}^2_f(K',\tfp)_{\ip-tor} \rightarrow \mathscr{I}/\ip.
\end{align}

\begin{thm} \label{cup is alternating}
    $\cup_{CT}$ is non-degenerate and alternating.
\end{thm}
\begin{proof}
    This is a special case of \cite[Proposition 12.7.13.4]{Nekovar}.
\end{proof}

\subsection{Extended Mordell-Weil Group} 
Let $K'_\p := \prod_{w|\p} K'_w$. For every such $w$, $E/K'_w$ has split multiplicative reduction and hence we have Tate parametrization
\begin{align*}
    \Phi_{Tate} : \overline{\qp}^*/q_E^\z \rightarrow E(\overline{\qp})
\end{align*}
where $q_E \in F_\mathfrak{p}$ with $\mathrm{ord}_p(q_E)>0$. We write by an abuse of notation the direct sum
\begin{align*}
    \Phi_{Tate} : (K'_\mathfrak{p})^* \rightarrow \bigoplus_{w|\p} E(K'_w).
\end{align*}
Define the extended Mordell-Weil group as follows
\begin{align*}
    E^\dagger(K') := \{(P,\widetilde{P}) : P \in E(K'), \widetilde{P} = (\widetilde{P_w})_{w|\p} \in (K'_\mathfrak{p})^*, \Phi_{Tate}(\widetilde{P_w}) = (i_w(P))_{w|\p} \}.
\end{align*}
For $w|\p$, write $q_w$ for the element of $E^\dagger(K')$, which has 0 in the component of $E(K')$, $q_E$ in the $w$ component and 1 in the other. We have a short exact sequence
\begin{align*}
    0 \rightarrow \oplus_{w|\p} \z \rightarrow E^\dagger(K') \rightarrow E(K') \rightarrow 0
\end{align*}
where the first map sends the element with 1 in the $w$-component and 0 in the other to $q_w$ and the second is the projection map. Note that $K'_\p \cong K' \otimes_F F_\p$. When $Gal(K'/F)$ is non trivial, we let $Gal(K'/F)$ act on $K'_\p$ by its action on the first component and let it act on $E^\dagger(K')$ diagonally. Define a map 
\begin{align} \label{defn of iedagger}
    i_E^\dagger : E^\dagger(K') \rightarrow \widetilde{H}^1_f(K',V_f)
\end{align}
as follows. Let $(P,\widetilde{P}) \in E^\dagger(K'), \widetilde{P} = (\widetilde{P_w}) \in (K'_\mathfrak{p})^*$. Since $\Phi_{Tate}(\widetilde{P_w}) = i_w(P)$, we obtain $(\Phi_{Tate})_*(\gamma_{\widetilde{P_w}}) = \gamma_{i_w(P)}$. Recall from (\ref{comm diag kummer theory}) we also have that $res_w(\gamma_P) = i_w^{-1}(\gamma_{i_w(P)})$. So for all $w|\p$, we have 
\begin{align*}
   p_w^+(\gamma_{\widetilde{P_w}}) = res_w(\gamma_P). 
\end{align*}
So for every representative $\gamma_P^0 \in C^1_{cont}(G_{K',S'},V_f)$ and $\gamma_{\widetilde{P_w}}^0 \in C^1_{cont}(K'_w,L(1))$ of $\gamma_P$ and $\gamma_{\widetilde{P_w}}$, respectively there exists a unique $\epsilon_w^0 \in C^0_{cont}(K'_w,V_f)$ such that
\begin{align*}
    res_w(\gamma_P^0) = p_w^+(\gamma_{\widetilde{P_w}}^0) - \delta(\epsilon_w^0)
\end{align*}
where $\delta$ is the differential of $C^\bullet(K'_w,V_f).$ This $\epsilon_w^0$ is unique because $H^0(K'_w,V_f) = 0$. Now for all $w \nmid \p$ in $S'$, put $\gamma_{\widetilde{P_w}}^0 = 0$ and $\epsilon_w^0 = 0$. 
Now write 
\begin{align*}
    (P,\widetilde{P})^0 := \left( \gamma_P^0,(\gamma_{\widetilde{P_w}^0})_{w \in S'}, (\epsilon_w^0)_{w \in S'} \right) \in \widetilde{C}^1_f(G_{K',S'},V_f).
\end{align*}
Finally define $i_E^\dagger((P,\widetilde{P}))$ as the cohomology class of $(P,\widetilde{P})^0$. It can be checked easily that this map is well-defined using the definitions.

\begin{lem} \label{iedagger is an iso}
    Consider the map $i_E^\dagger : E^\dagger(K') \otimes L \rightarrow \widetilde{H}^1_f(K',V_f)$, induced by $i_E^\dagger$. Then $i_E^\dagger$ is injective. Moreover, if we assume that $\Varpi(E/K')_{p^\infty}$ is finite, then $i_E^\dagger$ is an isomorphism.
\end{lem}

\begin{proof}
    We have the following commutative diagram with exact rows
    $$
    \begin{tikzcd}
	   0 \arrow[r] & 
      \oplus_{v|\p} L \arrow[r] \arrow[d,equal] &
	   E^\dagger(K') \otimes L \arrow[r,'] \arrow[d,"i_E^\dagger"] &
	   E(K') \otimes L \arrow[r] \arrow[d,"\kappa"] &
	   0
	   \\
	   0 \arrow[r] &
	   \oplus_{v|\p} L \arrow[r] &
	   \widetilde{H}^1_f(K',V_f) \arrow[r] &
	   H^1_f(K',V_f) \arrow[r] &
	   0.
    \end{tikzcd}
    $$
    Since the Kummer map $\kappa$ is injective, by the Snake lemma, $i_E^\dagger$ is injective. Now assume that $\Varpi(E/K')_{p^\infty}$ is finite. In order to show that $i_E^\dagger$ is surjective, it suffices to show that $coker(\kappa) = 0$. But we know that $coker(\kappa) = Ta_p(\Varpi(E/K')) \otimes L$ which is zero because $\Varpi(E/K')_{p^\infty}$ is finite. 
\end{proof}

\section{A $p$-adic Weight Pairing} \label{p-adic weight pairing section}
In this section, we introduce a $p$-adic weight pairing, which is used to relate the first cohomology group of the Selmer complex with the strict Selmer group (Corollary \ref{NonDegeneracyEquivalence}). Moreover, we use this relation to determine the structure of the strict Selmer group in Theorem \ref{ineq4}.

Recall from Section \ref{sec 6} that $\chi \in \{1,\epsilon_K\}$ is a quadratic character of $F$ of conductor coprime with $\mathfrak{np}$ and $K' \in \{F,K\}$ is the extension of $F$ associated with $\chi$. Define the strict Selmer group of $\T_\f/K'$ as
\begin{align*}
    Sel^{cc}_{Gr}(\f/K') := ker \left(H^1(G_{K',S'},\T_\f \otimes_\I \I^*) \rightarrow \prod_{w|\mathfrak{p}} H^1(K'_w,\T_{\f,w}^- \otimes_\I \I^*) \right)
\end{align*}
and denote by $X^{cc}_{Gr}(\f/K') := \mathrm{hom}_{\zp}(Sel^{cc}_{Gr}(\f/K'),\qp/\zp)$ its Pontryagin dual.
\\
For every $\z[Gal(K'/F)]$-module $M$, write $M^\chi$ for the submodule of $M$ on which $Gal(K'/F)$ acts via $\chi$, so that if $\chi$ is trivial, $M^\chi = M$ and otherwise $M^\chi = \{m \in M | \sigma \cdot m = -m \}$, where $\sigma$ is the nontrivial element of $Gal(K/F)$. The aim of this section is to prove the following theorem.

\begin{thm} \label{ineq4}
 Let $\chi \in \{1, \epsilon_K\}$ be a quadratic character of $F$ as above. Assume that 
 \begin{enumerate}
     \item $\chi(\p) = 1$,
     \item $\mathrm{rank}_\z \hspace{0.01mm} \hspace{1mm} E(K')^\chi = 1$, and
     \item $\Varpi(E/K')^\chi_{p^\infty}$ is finite.
 \end{enumerate}
 Then
 \begin{align*}
     X^{cc}_{Gr}(\f/K')^\chi \otimes_\I \I_P \cong \I_P/P\I_P.
 \end{align*}
\end{thm}
We prove this using the $p$-adic weight pairing defined in the next section.

\subsection{Definition of $p$-adic Weight Pairing} 
Define the $p$-adic weight pairing as
\begin{align*} 
    \langle -,- \rangle^{Nek}_{K'} : \widetilde{H}^1_f(K',V_f) \times \widetilde{H}^1_f(K',V_f) \rightarrow \ip/\varpi \xrightarrow[\varphi]{\sim} L \\
    \langle x,y \rangle^{Nek}_{K'} :=  l_{wt} \cdot \varphi(i_P(x) \cup_{CT} i_P(y))
\end{align*}
where $l_{wt} := \log_p(\gamma)$, $i_P$ is the map in Proposition \ref{SesForSelmerComplex} and $\cup_{CT}$ is as in (\ref{Cassels Tate Pairing}).

\begin{prop} \label{NonDegeneracyOfPairing}
    \begin{enumerate}
        \item $\langle x,x \rangle^{Nek}_{K'} = 0$ for every $x \in \widetilde{H}^1_f(K',V_f)$.

        \item $\langle -,- \rangle^{Nek}_{K'}$ is a $Gal(K'/F)$-equivariant pairing.
    \end{enumerate}
\end{prop}

\begin{proof}
    \begin{enumerate}
        \item By Theorem \ref{cup is alternating}, the cup-product $\cup_{CT}$ is an alternating pairing and hence, the assertion follows.
        \item This follows from the fact that both $i_P$ and $\cup_{CT}$ are $Gal(K'/F)$-equivariant.
    \end{enumerate}
\end{proof}

Write $\langle -,- \rangle^{Nek}_{K',\chi}$ for the restriction of $\langle -,- \rangle^{Nek}_{K'}$ to $\widetilde{H}^1_f(K',V_f)^\chi \times \widetilde{H}^1_f(K',V_f)^\chi$.
\\
If $M$ is an $\I$-module, we say that $M$ is semi-simple at $P$ if $M_P$ is a semi-simple $\ip$-module and recall, we write $\mathrm{len}_P(M)$ to denote the length of $M_P$ as an $\ip$-module.

\begin{prop}
    Let $\chi \in \{1,\epsilon_K\}$ be a quadratic character of conductor coprime to $\n\p$ such that $\chi(\p) = 1$. Then $\langle -,- \rangle^{Nek}_{K',\chi}$ is non-degenerate if and only if $\widetilde{H}^2_f(K',\tfp)^\chi$ is a torsion, semisimple $\ip$-module.
\end{prop}

\begin{proof}
    Write $M := \widetilde{H}^2_f(K',\tfp)_{\ip-tors}$. Then $M$ is a finitely generated module over the discrete valuation ring $\ip$. So by the structure theorem of finitely generated modules over principal ideal domains, we have an isomorphism of $\ip$-modules
        \begin{align*}
            M \xrightarrow{\sim} \oplus_{j = 0}^{n} \ip/(\varpi)^{e_j} 
        \end{align*}
        for $1 \leq e_1 \leq \ldots \leq e_n$. Write $\cup'_{CT}$ for the restriction of $\cup_{CT}$ to ${M^\chi[\varpi] \times M^\chi[\varpi]}$.
        \\
        \textbf{Claim 1} : $\langle -,- \rangle^{Nek}_{K',\chi}$ is non-degenerate if and only if $\cup'_{CT}$ is non-degenerate and $\widetilde{H}^1_f(K',\tfp)^\chi = 0$ or equivalently, using the isomorphism (\ref{H^1-H^2 relation}), $\cup'_{CT}$ is non-degenerate and $\widetilde{H}^2_f(K',\tfp)^\chi$ is a torsion $\ip$-module.
        \\
        Suppose $\cup'_{CT}$ is non-degenerate and $\widetilde{H}^1_f(K',\tfp)^\chi = 0$. Let $\langle x,y \rangle^{Nek}_{K',\chi} = 0$ for all $y \in \widetilde{H}^1_f(K',V_f)^\chi$. So, $i_P(x) \cup'_{CT} i_P(y) = 0$ for every $y$. Since $\cup'_{CT}$ is non-degenerate, we have $x \in ker(i_P) = \widetilde{H}^1_f(K',\T_\f)/\varpi = 0$. Conversely, let $\langle -,- \rangle^{Nek}_{K',\chi}$ be non-degenerate and suppose $x' \cup'_{CT} y' = 0$ for every $y' \in M^\chi[\varpi]$. Since $i_P$ is surjective, there exist $x,y \in \widetilde{H}^1_f(K',\tfp)^\chi$ such that $i_P(x) = x'$ and $i_P(y) = y'$. So we have that $\langle x,y \rangle^{Nek}_{K',\chi} = i_P(x) \cup'_{CT} i_P(y) = 0$ for every $y$. By non-degeneracy of $\langle -,- \rangle^{Nek}_{K',\chi}$, we get that $x' = 0$. Moreover, using similar arguments we get that $i_P$ is injective. Now since $ker(i_P) = \widetilde{H}^1_f(K',\tfp)^\chi/\varpi$ and $\widetilde{H}^1_f(K',\tfp)$ is a free $\ip$-module, we get that $\widetilde{H}^1_f(K',\tfp)^\chi = 0$.
        \\
        
        \textbf{Claim 2} :  $\cup'_{CT}$ is non-degenerate if and only if $e_j = 1$ for all $j$, or, in other words, $M$ is semisimple over $\ip$.
         Note that $ker(\cup'_{CT}) = \varpi M^\chi \cap M^\chi[\varpi]$. First, suppose that $\cup'_{CT}$ is non-degenerate, i.e., $\varpi M^\chi \cap M^\chi[\varpi] = 0$ and there is a $j$ such that $e_j > 1$. Fix such a $j$. Choose $r = (0,\ldots,\varpi^{e_j-1} \hspace{1mm} mod \hspace{1mm} \varpi^{e_j},\ldots 0) \in \oplus_{i = 0}^{n} \ip/(\varpi)^{e_i}$. Then $r \neq 0$. Now clearly, $\varpi r = 0$ and $r = \varpi(0,\ldots,\varpi^{e_j-2} \hspace{1mm} mod \hspace{1mm} \varpi^{e_j},\ldots 0)$. So $0 \neq x \in \varpi M^\chi \cap M^\chi[\varpi]$, which is a contradiction. Converse follows easily.
\end{proof}

\begin{prop} \label{length dim equality}
    $\mathrm{len}_P(\widetilde{H}^2_f(K',\T_\f)^\chi) \geq \mathrm{dim}_L(\widetilde{H}^1_f(K',V_f)^\chi)$ if and only if $\widetilde{H}^2_f(K',\T_{\f,P})^\chi$ is a torsion $\I_P$-module. In addition to this, if $\widetilde{H}^2_f(K',\T_{\f,P})^\chi$ is a semi-simple $\I_P$-module then we have an equality.
\end{prop}

\begin{proof}
    Write $N_* := \widetilde{H}^*_f(K',\T_{\f,P})^\chi$ and $\mathcal{N}_* := \widetilde{H}^*_f(K',V_f)^\chi$. By Proposition \ref{SesForSelmerComplex}, we have a short exact sequence of $L$-modules
    \begin{align*}
        0 \rightarrow N_q/\vp \rightarrow \mathcal{N}_q \rightarrow N_{q+1}[\vp] \rightarrow 0.
    \end{align*}
    Since $N_2$ is a torsion $\I_P$-module, $N_1 = 0$ (using the isomorphism in (\ref{H^1-H^2 relation})). So $\mathcal{N}_1 \cong N_2[\vp]$ and therefore, their $L$-dimensions are equal. By the structure theorem of finitely-generated, torsion modules over a PID, 
    \begin{align*}
        N_2 = \oplus_{j=1}^{\infty} (\I_P/\vp^j)^{m(j)}
    \end{align*}
    where $m(j) = 0$ for all $j >> 0$. This gives us
    \begin{align*}
        N_2[\vp] \cong \oplus_{j=1}^{\infty} (\I_P/\vp)^{m(j)}.
    \end{align*}
    So,
    \begin{align*}
        \mathrm{len}_P N_2 = \sum_{j = 0}^{\infty} m(j) \cdot j & = \sum_{j = 1}^{\infty} m(j) + \sum_{j = 2}^{\infty} m(j)(j-1) \\
        & = \mathrm{dim}_L N_2[\vp] + \sum_{j = 2}^{\infty} m(j)(j-1) \\
        & = \mathrm{dim}_L \mathcal{N}_1 + \sum_{j = 2}^{\infty} m(j)(j-1).
    \end{align*}
    In particular, $\mathrm{len}_P N_2 \geq \mathrm{dim}_L \mathcal{N}_1$. Moreover, $\mathrm{len}_P N_2 = \mathrm{dim}_L \mathcal{N}_1$ if and only if $m(j) = 0$ for every $j \geq 2$ if and only if $N_2$ is semi-simple also. 
\end{proof}

\begin{lem} \label{lem7.5}
    $H^2(F_\p,\T_\f^+ \otimes_\I \I_P) \cong \I_P/P\I_P$.
\end{lem}

\begin{proof}
     Since $\T_\f^+ \otimes_\I \I_P/\vp \cong L(1)$, there are short exact sequences
    \begin{align*}
        0 \rightarrow H^j(F_\p,\T_\f^+ \otimes_\I \I_P)/\vp \rightarrow H^j(F_\p,\qp(1)) \otimes L \rightarrow H^{j+1}(F_\p,\T_\f^+ \otimes_\I \I_P)[\vp] \rightarrow 0.
    \end{align*}
    Since $H^0(F_\p,\qp(1)) = 0$, the above sequence gives $H^1(F_\p,\T_\f^+ \otimes_\I \I_P)[\vp] = 0$ and hence $H^1(F_\p,\T_\f^+ \otimes_\I \I_P)$ is a free $\I_P$-module. Recall that $\T_\f^+ \cong \I(\ma_\p^{*-1}\chi_{cy}\chi_\Gamma^{1/2})$ and $\T_\f^- \cong \I(\ma_\p^*\chi_\Gamma^{-1/2})$, we have $H^0(F_\p,\T_\f^\pm) = 0$. So by the local Tate duality, $H^2(F_\p,\T_\f^+)$ is a torsion $\I$-module. By \cite[Corollary 4.6.10]{Nekovar} and \cite[5.2.11]{Nekovar}
    \begin{align} \label{SumOfRanks}
        \sum_{j=0}^{2} (-1)^j \mathrm{rank}_\I H^j(F_\p,\T_\f^+) = -2.
    \end{align}
    Since $H^0(F_\p,\T_\f^+) = 0$ and $H^2(F_\p,\T_\f^+)$ is a torsion $\I$-module and hence has $\I$-rank 0, we get using (\ref{SumOfRanks}) that $H^1(F_\p,\T_\f^+)$ has rank 2 as an $\I$-module and hence, $H^1(F_\p,\T_\f^+ \otimes_\I \I_P) \cong \I_P^2$. So we get the following exact sequences
    \begin{align*}
        0 \rightarrow (\I_P/\vp)^2 \rightarrow H^1(F_\p,\qp(1)) \otimes L \rightarrow H^2(F_\p,\T_\f^+ \otimes \I_P)[\vp] \rightarrow 0 \\
        0 \rightarrow H^2(F_\p,\T_\f^+ \otimes_\I \I_P)/\vp \rightarrow H^2(F_\p,\qp(1)) \otimes L \rightarrow 0.
    \end{align*}
    By \cite[Section 11.3]{Nekovar} $\mathrm{dim}_{\qp} H^1(F_\p,\qp(1)) = 3$ and $\I_P/\vp \cong L$, we get that $\mathrm{dim}_L H^2(F_\p,\T_\f^+ \otimes \I_P)[\vp] = 1$. Also, again by \cite[Section 11.3]{Nekovar} since $\mathrm{dim}_{\qp} H^2(F_\p,\qp(1)) = 1$, we have that $\mathrm{dim}_L H^2(F_\p,\T_\f^+ \otimes \I_P)/\vp = 1$. So by the structure theorem for finitely generated torsion modules over PID, we have that there is an $n \geq 1$ such that
    \begin{align*}
        H^2(F_\p,\T_\f^+ \otimes \I_P) \cong \I_P/\vp^n.
    \end{align*}
    In order to show that $n=1$, clearly it suffices to show that the map
    \begin{align*}
        \mathcal{F} : H^1(F_\p,L(1)) \twoheadrightarrow H^2(F_\p,\T_\f^+ \otimes_\I \I_P) [\vp] \hookrightarrow H^2(F_\p,\T_\f^+ \otimes_\I \I_P) \\
        \twoheadrightarrow H^2(F_\p,\T_\f^+ \otimes_\I \I_P)/\vp \xrightarrow{\sim} H^2(F_\p,L(1)) \xrightarrow[inv_\p]{\sim} L
    \end{align*}
    is non-zero.
    \\
    By Kummer theory, we have $H^1(F_\p,L(1)) \cong F_\p^\times \otimes L$. Let $q \in F_\p^\times$ be an arbitrary element. We show that $\mathcal{F}(q) := \mathcal{F}(q \otimes 1) \in L^\times$. Write $c_q : G_{F_\p} \rightarrow L(1)$ for a 1-cocycle representing $q \otimes 1$. Recall the surjective morphism $P : \I_P \twoheadrightarrow L$ and consider $c_q : G_{F_\p} \rightarrow \T_\f^+ \otimes \I_P$ as a 1-cochain that lifts $c_q$ under $P$. In $C^\bullet_{cont}(F_\p,\T_\f^+ \otimes \I_P)$ we have
    \begin{align*}
        d c_q(g,h) = \ma_\p^*(g)^{-1}\chi_{cy}(g)\chi_\Gamma(g)^{1/2}c_q(h) - c_q(gh) + c_q(g).
    \end{align*}
    Since $c_q(gh) = c_q(g) - \chi_{cy}(g)c_q(h)$, we get that 
    \begin{align*}
        d c_q(g,h) = \chi_{cy}(g)(\ma_\p^*(g)^{-1}\chi_\Gamma(g)^{1/2} - 1)c_q(h).
    \end{align*}
    Let $\varrho(g,h) := \chi_{cy}(g)c_q(h)P \left(\frac{\ma_\p^*(g)^{-1}\chi_\Gamma^{1/2}(g) - 1}{\vp} \right) \in L(1)$. Then $\varrho$ is a 2-cocycle and $\mathcal{F}(q)$ is the image of the class of $\varrho$ under $inv_\p$. Let
    \begin{align*}
        \langle -,- \rangle_{Tate} : H^1(F_\p,L) \times H^1(F_\p,L(1)) \rightarrow L
    \end{align*}
    be the Tate local cup-product pairing. Then by Kummer theory, 
    \begin{align*}
        \mathcal{F}(q) = inv_\p(class \hspace{1mm} of \hspace{1mm} \varrho) = \langle \Phi,q \rangle_{Tate} \in L, 
    \end{align*}
    where $\Phi := P \left(\frac{\ma_\p^{*-1}\chi_\Gamma^{1/2} - 1}{\vp} \right)$.
    \\
    Let $I_\p$ denote the inertia group at $\p$ and let $g' \in I_\p$. Then the image of $g'$ under $\chi_{cy}^{1/2}$ in $\Gamma$ is $\gamma^z$ for some $z \in \zp$ such that $\frac{1}{2} log_p(\chi_{cy}(g')) = z log_p(\gamma)$. Since $\ma_\p^*$ is an unramified character, $\ma_\p^*(g') = 1$ and hence
    \begin{align*}
        \Phi(g) & = P \left(\frac{\ma_\p^{*-1}(g')\chi_\Gamma(g') - 1}{\vp} \right) = P \left(\frac{\gamma^z - 1}{\gamma - 1} \right) = M \left(\frac{\gamma^z - 1}{\gamma - 1} \right)(2) = z.
    \end{align*}
    Let $rec_\p : F_\p^* \rightarrow G_\p^{ab}$ be the reciprocity map of local class field theory and let $\psi$ denote the isomorphism $G_\p^{ab} \xrightarrow{\sim} G_\p^{ur} \times Gal(F_\pi/F_\p)$, where $G_\p^{ur}$ is the Galois group of the maximal unramified extension $F_\p^{ur}$ of $F_\p$ and $F_\pi$ is an abelian extension of $F_\p$ with $Gal(F_\pi/F_\p) \cong O_\p^*$, with $O_\p^*$ the unit group of ring of integers of $F_\p$. (For details, check \cite[Section 3]{Serre}). If $q = \p^{\mathrm{ord}_\p(\alpha)}u \in F_\p^\times$ then $\psi(rec_\p(q)) = g^{-1} Fr_\p^{\mathrm{ord}_\p(q)}$, where $g \in Gal(F_\pi/F_\p)$ is the element that maps to $u \in O_\p^*$ under the isomorphism $Gal(F_\pi/F_\p) \cong O_\p^*$ and $Fr_\p \in G_\p^{ur}$ is an arithmetic Frobenius. By an abuse of notation, we again write $rec_\p$ for the composition of $rec_\p$ with $\psi$. Now
    \begin{align*}
        \Phi(Fr_\p^n) & = P \left(\frac{\ma_p^{*-1}(Fr_\p^n) - 1}{\vp} \right) = P \left(\frac{\ma_\p^{-n} - 1}{\vp} \right) = \frac{1}{2} \frac{n}{l_{wt}} \mathscr{L}_p(E),
    \end{align*}
    where $\mathscr{L}_p(E) := f_{\p/p} \frac{log_p(N(q_E))}{\mathrm{ord}_p(N(q_E))}$ is the $L$-invariant of $E/F_\p$ (here, $N := N_{F_\p/\qp}$ is the usual field norm and $q_E$ is the Tate period associated with $E/F_\p$). The last equality follows using \cite[Proposition 8.7]{MokExceptionalZero} that $\frac{d}{dk}\ma_\p(k)|_{k=2} = \mathscr{L}_p(E)$.  
    \\
    Now combining the above calculations we get that
    \begin{align*}
        \Phi(rec_\p(q)) = \Phi(g^{-1}Fr_\p^{\mathrm{ord}_\p(q)}) & = \frac{-1}{2} \frac{log_p(\chi_{cy}(g))}{l_{wt}} + \frac{1}{2} \frac{\mathrm{ord}_\p(q)}{l_{wt}} \mathscr{L}_p(E) \\
        & = \frac{-1}{2} \frac{1}{l_{wt}} \left(log_p(N(u)) - \mathrm{ord}_p(N(q)) \frac{log_p(N(q_E))}{\mathrm{ord}_p(N(q_E))} \right) \\
        & = \frac{-1}{2} \frac{1}{l_{wt}} \left(log_p(N(q)) - \mathrm{ord}_p(N(q)) \frac{log_p(N(q_E))}{\mathrm{ord}_p(N(q_E))} \right) \\
        & = \frac{-1}{2} \frac{1}{l_{wt}} log_{N(q_E)}(q). 
    \end{align*}
    So, $\mathcal{F}(q) = \langle \Phi,q \rangle_{Tate} = \Phi(rec_\p(q)) = \frac{-1}{2} \frac{1}{l_{wt}} log_{N(q_E)}(q) \neq 0$.
\end{proof}

\begin{lem}
    Assume that $\widetilde{H}^2_f(K',\T_{\f,P})^\chi$ is a torsion, semi-simple $\I_P$-module. Then $X^{cc}_{Gr}(\f/K')^\chi \otimes_\I \I_P$ is a torsion, semi-simple $\I_P$-module and 
    \begin{align*}
        \mathrm{len}_P(X^{cc}_{Gr}(\f/K')^\chi) = \mathrm{dim}_L(H^1_f(K',V_f)^\chi).
    \end{align*}
\end{lem}

\begin{proof}
     Since $adj(\pi) : \T_\f \cong \mathrm{hom}_\I(\T_\f,\I(1))$ and $\T_\f$ is a free $\I$-module, we have an isomorphism of $\I[G_{K',S'}]$-modules
    \begin{align*}
        \T_\f \otimes_\I \I^* \cong \mathrm{hom}_\I(\T_\f,\I(1)) \otimes_\I \mathrm{hom}_{cont}(\I,\qp/\zp) \cong \mathds{A}_\f. 
    \end{align*}
    Similarly, $\T_\f^- \otimes_\I \I^* \cong \mathds{A}_\f^-$. So,
    \begin{align*}
        Sel^{cc}_{Gr}(\f/K') \cong ker \left(H^1(G_{K',S'},\mathds{A}_\f) \rightarrow \prod_w H^1(K'_w,\mathds{A}_\f^-) \right).
    \end{align*}
    Using \cite[6.1.3.2]{Nekovar}, we have an exact sequence of $\I$-modules
    \begin{align}\label{ExactSequenceOfSelcc}
        H^0(G_{K',S'},\mathds{A}_\f) \rightarrow \oplus_{w|\p} H^0(K'_w,\mathds{A}_{\f,w}^-) \rightarrow \widetilde{H}^1_f(K',\mathds{A}_\f) \rightarrow Sel^{cc}_{Gr}(\f/K') \rightarrow 0.
    \end{align}
    For any $\I$-module $M$, write $M_P^*$ for the localization at $P$ of its Pontryagin dual, i.e., 
    \begin{align*}
        M_P^* := \mathrm{hom}_{\zp}(M,\qp/\zp) \otimes_\I \I_P.
    \end{align*}
    By local Tate duality, $H^2(K'_w,\T_\f) \cong \mathrm{hom}_{\zp}(H^0(K'_w,\mathds{A}_\f),\qp/\zp)$, where $w$ is a prime of $K'$. Since $H^0(G_{K',S'},\mathds{A}_f) \subset H^0(K'_w,\mathds{A}_\f)$, we have a surjection $H^2(K'_w,\T_{\f,P}) \twoheadrightarrow H^0(G_{K',S'},\mathds{A}_\f)^*_P$. By \cite[Proposition 12.7.13.3]{Nekovar}, $\widetilde{\textbf{R}\Gamma_{cont}}(K'_w,\T_{\f,P})$ is acyclic for every $w \nmid p$, so in particular, $H^2(K'_w,\T_{\f,P}) = 0$ and hence, $H^0(G_{K',S'},\mathds{A}_\f)^*_P = 0$. By \cite[Section 6.3]{Nekovar}, $\widetilde{H}^1_f(K',\mathds{A}_\f)^* \cong \widetilde{H}^2_f(K',\T_\f)$. Taking Pontryagin duals and then localizing (\ref{ExactSequenceOfSelcc}) at $P$, we get the following exact sequence
    \begin{align*}
        0 \rightarrow X^{cc}_{Gr}(\f/K') \otimes_\I \I_P \rightarrow \widetilde{H}^2_f(K',\T_{\f,P}) \rightarrow \oplus_{w|\p} H^2(K'_w,\T_{\f,w}^+ \otimes_\I \I_P) \rightarrow 0.
    \end{align*}
     By Lemma \ref{lem7.5}, we have an isomorphism of $\I_P$-modules
    \begin{align*}
       H^2(K'_w,\T_{\f,w}^+ \otimes_\I \I_P) \cong H^2(F_\p,\T_\f^+ \otimes_\I \I_P) \cong \I_P/P\I_P
    \end{align*}
    for every $w|\p$. If $\chi$ is trivial then $K'=F$ and the only prime of $K'$ dividing $\p$ is $\p$. If $\chi$ is non trivial, let $w_1,w_2$ be the primes of $K'$ dividing $\p$. In this case the non trivial element of $Gal(K/F)$ acts by permuting the factors in the sum
    \begin{align*}
        H^2(K_{w_1},\T_{\f,{w_1}}^+ \otimes_\I \I_P) \oplus  H^2(K_{w_2},\T_{\f,{w_2}}^+ \otimes_\I \I_P).
    \end{align*}
    So, in both the cases, the $\chi$-component of $\oplus_{w|\p} H^2(K'_w,\T_{\f,w}^+ \otimes_\I \I_P)$ is isomorphic to $H^2(K'_w,\T_{\f,w}^+ \otimes_\I \I_P)$, where $w \in \{w_1,w_2\}$.
    So, in either case, taking the $\chi$-component in the previous exact sequence gives us
    \begin{align*}
         0 \rightarrow X^{cc}_{Gr}(\f/K')^\chi \otimes_\I \I_P \rightarrow \widetilde{H}^2_f(K',\T_{\f,P})^\chi \rightarrow \I_P/P\I_P \rightarrow 0.
    \end{align*}
    So if $\widetilde{H}^2_f(K',\T_{\f,P})^\chi$ is a torsion, semi-simple $\I_P$-module then so is $X^{cc}_{Gr}(\f/K')^\chi \otimes_\I \I_P$. Moreover, the above short exact sequence gives us
    \begin{align*}
        \mathrm{len}_P(X^{cc}_{Gr}(\f/K')^\chi) & = \mathrm{len}_P(\widetilde{H}^2_f(K',\T_\f)^\chi) - \mathrm{len}_P(\I_P/P\I_P) \\
        & = \mathrm{dim}_L(\widetilde{H}^1_f(K',V_f)^\chi) - 1 && (\text{Proposition} \hspace{1mm} \ref{length dim equality}) \\
        & = \mathrm{dim}_L (H^1_f(K',V_f)^\chi) && (\text{Proposition} \hspace{1mm} \ref{SesForSelmerComplex}(4)).
    \end{align*}
\end{proof}

So to summarize, we have proved :
\begin{cor} \label{NonDegeneracyEquivalence}
    Let $\chi \in \{1, \epsilon_K\}$ be a quadratic character of $F$ with conductor coprime to $\n\p$ such that $\chi(\p) = 1$. Then  the following are equivalent
    \begin{enumerate}
        \item $\langle -,- \rangle^{Nek}_{K',\chi}$ is non-degenerate.
        \item $\widetilde{H}^2_f(K',\T_{\f,P})^\chi$ is a torsion, semi-simple $\ip$-module.
        \item $\mathrm{len}_P(\widetilde{H}^2_f(K',\T_{\f,P})^\chi) = \mathrm{dim}_L(\widetilde{H}^1_f(K',V_f)^\chi)$.
    \end{enumerate}
    Moreover, if this is the case then, $X_{Gr}^{cc}(\f/K')^\chi_P$ is also a torsion, semi-simple $\ip$-module and
    \begin{align*}
        \mathrm{len}_P(X^{cc}_{Gr}(\f/K')^\chi) = \mathrm{dim}_L(H^1_f(K',V_f)^\chi).
    \end{align*}
\end{cor}

\subsection{Some $p$-adic Weight Pairing Computations}
Identify $\I_P/\vp \cong L$ as earlier and for every $x \in \I_P$, write
\begin{align*}
    x(0) & = x \hspace{1mm} mod \hspace{1mm} \vp \in L \\
    \frac{dx}{d\vp} & := (\vp^{-1}(x-x(0))) \hspace{1mm} mod \hspace{1mm} \vp \in L.
\end{align*}
Note that since every element of $\Gamma$ is of the form $\gamma^z$ for some $z \in \zp$, we easily get that for every $x \in \Gamma, \frac{dx}{d\vp} = \frac{\log_p(x)}{l_{wt}}$. Moreover, $\frac{d}{d\vp}$ satisfies the product rule: for every $A,B \in \I_P, \frac{dAB}{d\vp} = A(0)\frac{dB}{d\vp} + B(0)\frac{dA}{d\vp}$.
\\
Also, write $\Psi$ for the restriction of $\chi_\Gamma^{-1/2}$ to $G_\p$. Define the map
\begin{align*}
    \chi_E^{wt} := (\ma_\p^* \hspace{1mm} mod \hspace{1mm} \vp) \frac{d}{d\vp}(\Psi \cdot \ma_\p^*) \in \mathrm{hom}_{cont}(G_\p,L)
\end{align*}
and for a prime $w$ of $K'$ dividing $\p$, write $\chi_w^{wt} := Res_{K'_w/F_\p}(\chi_E^{wt}) \in \mathrm{hom}_{cont}(G_w,L)$ for the restriction of $\chi_E^{wt}$ to $G_w$. Then $\chi_w^{wt}$ defines a 1-cochain. Let $inv_{K'_w} : H^2(K'_w,L(1)) \xrightarrow{\sim} L$ be the invariant map of class field theory and define the perfect duality $\langle -,- \rangle_{K'_w} : H^1(K'_w,L(1)) \times H^1(K'_w,L) \rightarrow L$ as $\langle x,y \rangle_{K'_w} := inv_{K'_w}(x \cup y)$. Finally recall we have the following short exact sequence
\begin{align*}
    0 \rightarrow \oplus_{w|\p} L \rightarrow \widetilde{H}^1_f(K',V_f) \rightarrow H^1_f(K',V_f) \rightarrow 0.
\end{align*}
Write $Q_w \in \widetilde{H}^1_f(K',V_f)$ for the image of the element in $\oplus_{w|\p} L$ with 1 in the $w$-component and 0 in the other component under the above map.

\begin{lem} \label{lem7.7}
    For every $P_f = [(P,P^+,\epsilon_f)] \in \widetilde{H}^1_f(K',V_f)$ we have
    \begin{align*}
        \langle Q_w, P_f \rangle^{Nek}_{K'} = l_{wt} \cdot \langle [P_w^+],\chi_w^{wt} \rangle_{K'_w},
    \end{align*}
    where $[P_w^+] \in H^1(K'_w,L(1))$ is the cohomology class of the $w$-component of $P^+ \in \oplus_{w|\p} C^1_{cont}(K'_w,L(1))$.
\end{lem}

\begin{proof}
    Write $\ox := i_P(Q_w)$ and $\oy := i_P(P_f)$. Fix a splitting of $L$-modules (resp. $\I$-modules) $V_f \xrightarrow{\sim} V_f^+ \oplus V_f^- \xrightarrow{\sim} L^2$ (resp. $\T_\f \xrightarrow{\sim} \T_\f^+ \oplus \T_\f^- \xrightarrow{\sim} \I^2$) with $G_w$-action given by
    $\begin{pmatrix}
        \chi_{cy} & \star' \\ 0 & 1
    \end{pmatrix}$
    $\left(resp. 
    \begin{pmatrix}
        \chi_{cy} \Psi^{-1/2} \ma_\p^{*-1} & \star \\ 0 & \Psi \ma_\p^*
    \end{pmatrix}
    \right)$.
    Write $mod \hspace{1mm} \vp$ for the composition $\T_\f^? \twoheadrightarrow T_\f^? \otimes \I_P/P\I_P \xrightarrow{\sim} V_f^?$ for every $? \in \{\phi,+,-\}$. Now $Q_w = [(0,\star',(0,1))] \in \widetilde{H}^1_f(K',V_f)$. Let $\widetilde{\chi}_w^{wt} \in C^1_{cont}(K'_w,\T_\f)$ be the 1-cochain defined by
    \begin{align*}
        \widetilde{\chi}_w^{wt}(g) := (0, \vp^{-1}(1-\Psi(g)\ma_\p^*(g)) \in \T_\f^+ \oplus \T_\f^-.
    \end{align*}
      Note that $p_w^-(\widetilde{\chi}_w^{wt} \hspace{1mm} mod \hspace{1mm} \vp) = \chi_w^{wt}$ for the projection $p_w^- : C^1_{cont}(K'_w,V_f) \rightarrow C^1_{cont}(K'_w,L)$. Now for $X \in \{\T_\f, V_f\}$, we write $\widetilde{C}^\bullet_f(X) := \widetilde{C}^\bullet_f(G_{K',S'},X)$. Then the following diagram commutes: 
    \begin{equation}
\begin{tikzcd}
    0 \arrow[r] & \widetilde{C}^1_f(\T_\f) \arrow[r,"\vp"] & \widetilde{C}^1_f(\T_\f) \arrow[r] \arrow[dr, "d"'] & \widetilde{C}^1_f(V_f) \arrow[r, "i_P"] & \widetilde{C}^2_f(\T_\f) \arrow[dl, "\vp"] \\
    & & & \widetilde{C}^2_f(\T_\f) 
\end{tikzcd}
\end{equation}
So $i_P(Q_w)$ is represented by a 2-cocycle $c \in \widetilde{C}^1_f(\T_\f)$ such that $\vp \cdot c = d_{\widetilde{C}^\bullet_f}(\wx)$ where $\wx \in \widetilde{C}^1_f(\T_\f)$ is any 1-cochain which lifts a representative of $Q_w$ under $\widetilde{C}^1_f(\T_\f) \rightarrow \widetilde{C}^1_f(V_f)$. Since the splittings are compatible with $mod \hspace{1mm} \vp$, we can take
    \begin{align*}
        \wx := (0,\star,(0,1)) \in \widetilde{C}^1_f(\T_\f).
    \end{align*}
     Now we calculate $\vp^{-1}d_{\widetilde{C}^\bullet_f}(\wx)$. Clearly, the first component is 0, we ignore the second component and write $\_$ in its place as this will not appear in our calculations ahead and the third component is the following cochain: for $g \in G_w$
    \begin{align*}
        g & \mapsto (res_w - i_w^+)((0,\star))(g) + \delta((0,1))(g)\\
        & = g((0,1)) - (0,1) - i_w^+(\star)(g) \\
        & = (\star(g),\Psi(g)\ma_\p^*(g)) - (0,1) - (\star(g),0) \\
        & = (0,\Psi(g)\ma_\p^*(g)-1).
    \end{align*}
    So the third component is $(0,\vp^{-1} \cdot (\Psi(g)\ma_\p^*(g)-1)) = -\widetilde{\chi}_w^{wt}(g)$ and hence we get $c = (0,\_,-\widetilde{\chi}_w^{wt})$ (as mentioned earlier, we are ignoring the second component). In particular, $\ox = [0,\_,-\widetilde{\chi}_w^{wt})] \in \widetilde{H}^2_f(K',\T_\f)$. Under the surjective map $\widetilde{C}^1_f(K',\T_\f) \rightarrow \widetilde{C}^1_f(K',V_f)$, let
    \begin{align*}
        \wy := (\widetilde{P},\widetilde{P}^+,\widetilde{\epsilon}_P) \in \widetilde{C}^1_f(K',\T_\f)
    \end{align*}
    be a lift of $(P,P^+,\epsilon_P).$ We claim that it suffices to prove that 
    \begin{align} \label{ClaimInLemma7.7}
        \langle Q_w,P_f \rangle^{Nek}_{K'} = l_{wt} \cdot inv_w(L)\left([(-\widetilde{\chi}_w^{wt} \cup_\pi p_w^+(\widetilde{P}_w^+)) \hspace{1mm} mod \hspace{1mm} \vp] \right),
    \end{align}
    with the notations as follows: $inv_w(L) : H^2(K'_w,L(1)) \xrightarrow{\sim} L$ is the isomorphism induced by the invariant maps $H^2(K'_w,\z/p^n\z(1)) \xrightarrow{\sim} \z/p^n\z$ and $\cup_\pi : \widetilde{C}^\bullet_f(K'_w,\T_\f) \otimes \widetilde{C}^\bullet_f(K'_w,\T_\f) \rightarrow \widetilde{C}^\bullet_f(K_w,\I(1))$ is the cup-product induced by $\pi: \T_\f \otimes \T_\f \rightarrow \I(1)$. Indeed we have that 
    \begin{align*}
        -\left(\widetilde{\chi}_w^{wt} \cup_\pi p_w^+(\widetilde{P}_w^+) \right) \hspace{1mm} mod \hspace{1mm} \vp = \left(-\widetilde{\chi}_w^{wt} \hspace{1mm} mod \hspace{1mm} \vp \right) \cup_W \left(p_w^+(P_w^+) \right),
    \end{align*}
    where recall from Section \ref{BigGaloisRepresentation} that $(\cdot,\cdot)_W : V_f \times V_f \rightarrow L$ is the Weil pairing on $E$.
    Moreover,
    \begin{align*}
        \rho_w(y \cup_W p_w^+(x)) = -p_w^-(y) \cup x \in C^2_{cont}(K'_w,L(1)).
    \end{align*}
    Then
    \begin{align*}
        \langle Q_w,P_f \rangle^{Nek}_{K'} & = l_{wt} \cdot inv_w(L)\left([(-\widetilde{\chi}_w^{wt} \cup_\pi p_w^+(\widetilde{P}_w^+)) \hspace{1mm} mod \hspace{1mm} \vp] \right) \\
        & = l_{wt} \cdot inv_w(L)\left(-\widetilde{\chi}_w^{wt} \hspace{1mm} mod \hspace{1mm} \vp \right) \cup_W \left(p_w^+(P_w^+) \right) \\
        & = l_{wt} \cdot inv_w(L)\left(-p_w^-(-\widetilde{\chi}_w^{wt} \hspace{1mm} mod \hspace{1mm} \vp) \cup P_w^+ \right) \\
        & = l_{wt} \cdot inv_w(L)\left(-\chi_w^{wt} \cup P_w^+ \right) \\
        & = l_{wt} \cdot inv_{K'_w}\left([-\chi_w^{wt} \cup P_w^+] \right) \\
        & = l_{wt} \cdot \langle [P_w^+],\chi_w^{wt} \rangle_{K'_w}.
    \end{align*}
    Now it remains to show the equality in (\ref{ClaimInLemma7.7}). Let
    \begin{align*}
        \ox' := (\ox, \wx \otimes \vp^{-1}) \in \left(\widetilde{C}^\bullet_f(K',\T_\f) \otimes_{\I_P} \overline{\I_P} \right)^2 = \widetilde{C}^2_f(K',\T_\f) \oplus \widetilde{C}^1_f(K',\T_\f) \otimes Q(\I_P)\\
        \oy' := (\oy, \wy \otimes \vp^{-1}) \in \left(\widetilde{C}^\bullet_f(K',\T_\f) \otimes_{\I_P} \overline{\I_P} \right)^2 = \widetilde{C}^2_f(K',\T_\f) \oplus \widetilde{C}^1_f(K',\T_\f) \otimes Q(\I_P)
    \end{align*}
    Under projection in the sequence (\ref{ses}), $\ox'$ (resp. $\oy'$) lifts a representative of $\ox$ (resp. $\oy$). Then
    \begin{align*}
        ((id \otimes v_{\ip}) \circ s_{23})(\ox' \otimes \oy') & = ((id \otimes v_{\ip}) \circ s_{23}))((\ox,\wx \otimes \vp^{-1}) \otimes (\oy,\wy \otimes \vp^{-1})) \\
        & = (\ox \otimes \oy, \ox \otimes \wy \otimes \vp^{-1}),
    \end{align*}
    where the notations are as in Section \ref{CT Pairing}.
    Now choosing $r=0$
     \begin{align*}
         (\cup_{\pi,0} \otimes id)((\ox \otimes \oy), (\ox \otimes \wy) \otimes \vp^{-1}) = ((0,-\widetilde{\chi}_w^{wt} \cup_\pi p_w^+(y^+)), (0,(-\widetilde{\chi}_w^{wt} \cup_\pi p_w^+(\widetilde{P_w}^+) \otimes \vp^{-1})),
     \end{align*}
     where $y^+$ is the second component of $\oy$.
     Write $\mathcal{X} := (0,-\widetilde{\chi}_w^{wt} \cup_\pi p_w^+(y^+))$ and $\mathcal{Y} := (0,-\widetilde{\chi}_w^{wt} \cup_\pi p_w^+(\tp_w^+))$. Then using \cite[Lemma 4.9]{venerucci2013p}, we get that
     \begin{align*}
         \langle Q_w,P_f \rangle^{Nek}_{K'} & = l_{wt} \cdot inv_w(L)\left([\mathcal{Y}\  mod \hspace{1mm} \vp] \right) \\
         & = l_{wt} \cdot inv_w(L)\left([(-\widetilde{\chi}_w^{wt} \cup_\pi p_w^+(\widetilde{P}_w^+)) \hspace{1mm}  \ mod \hspace{1mm} \vp] \right).
     \end{align*}
\end{proof}

Now recalling the morphism $i_E^\dagger$ from (\ref{defn of iedagger}), define $\langle x,y \rangle^{Nek}_{K'} := \langle i_E^\dagger(x),i_E^\dagger(y) \rangle^{Nek}_{K'}$ for $x,y \in E^\dagger(K')$.

\begin{thm} \label{PairingInTermsOfReciprocityMap}
    For every $(P,\widetilde{P}) \in E^\dagger(K')$ and $w|\p$, we have
    \begin{align*}
        \langle q_w,(P,\widetilde{P}) \rangle^{Nek}_{K'} = l_{wt} \cdot \chi_E^{wt}(rec_\p(N_{K'_w/F_\p}(\widetilde{P}_w))),
    \end{align*}
    where $\widetilde{P}_w$ is the $w$-component of $\widetilde{P}$ and $rec_\p : F_\p^* \rightarrow G_\p^{ab}$ is the reciprocity map of local class field theory.
\end{thm}

\begin{proof}
    Recall from (\ref{ses for Vf}) the short exact sequence $0 \rightarrow L(1) \rightarrow V_f \rightarrow L \rightarrow 0$ and let $\partial_w : L \rightarrow H^1(K'_w,L(1))$ be the connecting homomorphism attached to this sequence. Then under the notations of Section \ref{kummer}, we have that $\gamma_{q_E} = \partial_w(1)$. As earlier, fix a splitting $V_f \xrightarrow{\sim} L(1) \oplus L$ with $G_w$-action given by
    $\begin{pmatrix}
        \chi_{cy} & \star' \\
        0 & 1
    \end{pmatrix}$.
    Moreover, $\gamma_{q_E}$ is represented by the 1-cocycle $\star' \in C^1_{cont}(K'_w,L(1))$. By definition,
    \begin{align*}
        i_E^\dagger(q_w) = [(0,\star',(0,1))] = Q_w \in \widetilde{H}^1_f(K',V_f).
    \end{align*}
     As a consequence of Lemma \ref{lem7.7} and the fact that local Tate duality is compatible with restriction and co-restriction on cohomology we get that
    \begin{align*}
        \langle q_w, (P,\widetilde{P}) \rangle^{Nek}_{K'} & := \langle i_E^\dagger(q_w), i_E^\dagger((P,\widetilde{P}) \rangle^{Nek}_{K'} \\
        & = l_{wt} \cdot \langle \gamma_{\widetilde{P_w}},\chi_w^{wt} \rangle_{K'_w} \\
        & = l_{wt} \cdot \langle \gamma_{\widetilde{P_w}},Res_{K'_w/F_\p}(\chi_E^{wt}) \rangle_{K'_w} \\
        & = l_{wt} \cdot \langle Cor_{K'_w/F_\p}(\gamma_{\widetilde{P_w}}),\chi_E^{wt} \rangle_{K'_w} \\
        & = l_{wt} \cdot \langle \gamma_{N_{K'_w/F_\p}(\widetilde{P_w})},\chi_E^{wt} \rangle_{K'_w} \\
        & = l_{wt} \cdot \chi_E^{wt}(rec_\p(N_{K'_w/F_\p}(\widetilde{P}_w))).
    \end{align*}
\end{proof}

For any $q \in \overline{\qp}^*$ such that $\mathrm{ord}_p(q) \neq 0$, write $\log_q$ for the branch of the $p$-adic logarithm which vanishes at $q$:
\begin{align*}
    \log_q(x) := \log_p(x) - \frac{\log_p(q)}{\mathrm{ord}_p(q)} \mathrm{ord}_p(x).
\end{align*}

\begin{cor} \label{PairingInTermsOfLog}
    For every $(P,\widetilde{P}) \in E^\dagger(K')$ and $w|\p$, we have
    \begin{align*}
        \langle q_w,(P,\widetilde{P}) \rangle^{Nek}_{K'} = -\frac{1}{2} \log_{N(q_E)}(N_{K'_w/\qp}(\widetilde{P}_w)),
    \end{align*}
    where, as earlier, $N := N_{F_\p/\qp}$ is the usual field norm and $q_E \in F_\p^\times$ is the Tate period of $E/F_\p$. In particular, for $i,j = 1,2$
    \begin{align*}
        \langle q_{w_i},q_{w_j} \rangle^{Nek}_{K'} = 0
    \end{align*}
    where $w_1$ and $w_2$ are primes dividing $\p$.
\end{cor}

\begin{proof}
     Let $\mathrm{ord}_w : (K'_w)^* \twoheadrightarrow \z$ be the valuation attached to $w$ and let $O_w$ be the ring of integers in $K'_w$. Put
    \begin{align*}
        Q := \mathrm{ord}_w(q_E) \cdot P, & \hspace{5mm} \widetilde{Q}_w^* := q_E^{-\mathrm{ord}_w(\tp_w)}\tp_w^{\mathrm{ord}_w(q_E)} \in O_w^* \\
        \widetilde{Q}_u^* := \tp_u^{\mathrm{ord}_w(q_E)}, w \neq u|\p, & \hspace{5mm} \wq^* := (Q_u^*)_{u|\p}.
    \end{align*}
    Then $\mathrm{ord}_w(q_E) \cdot (P,\tp) = (Q,(\tp_u^{\mathrm{ord}_w(q_E)})_u) = (Q,\wq^*) + \mathrm{ord}_w(\tp_w) \cdot q_w$. Since $\langle -,- \rangle^{Nek}_{K'}$ is non-degenerate, we get that
    \begin{align*}
        \mathrm{ord}_w(q_E) \langle q_w,(P,\widetilde{P}) \rangle^{Nek}_{K'} = \langle q_w,(Q,\wq^*) \rangle^{Nek}_{K'}.
    \end{align*}
     Let $n_w := N_{K'_w/F_\p}(\wq_w^*) \in O_\p^*$. Then $rec_\p(n_w) \mapsto (1,n_w) \in G_\p^{ur} \times O_\p^*$. In particular, $\ma_\p^*(rec_\p(n_w)) = 1$ and $\Psi(rec_\p(n_w)) = [\kappa(N(n_w))^{-1/2}]$ (using \cite[Equation 8.13]{MokExceptionalZero}), where $\kappa$ is the projection of $\zp^*$ onto $\Gamma$. So using Theorem \ref{PairingInTermsOfReciprocityMap} we get that
    \begin{align*}
        \langle q_w, (Q,\wq^*) \rangle^{Nek}_{K'} = l_{wt} \cdot \frac{d}{d\vp} ([\kappa(N(n_w))^{-1/2}]) = l_{wt} \frac{\log_p[\kappa(N(n_w))^{-1/2}]}{l_{wt}} = -\frac{1}{2} \log_p(N(n_w)).
    \end{align*}
    The second equality follows from the note after the definition of $\frac{d}{d\vp}$. So,
    \begin{align*}
        \langle q_w,(P,\widetilde{P}) \rangle^{Nek}_{K'} & = \mathrm{ord}_w(q_E)^{-1} \cdot \langle q_w,(Q,\wq^*) \rangle^{Nek}_{K'} \\
        & = -\frac{1}{2\mathrm{ord}_w(q_E)} \log_p(N(N_{K'_w/F_\p}(\wq_w^*))) \\
        & = -\frac{1}{2} \left(\log_p(N(N_{K'_w/F_\p}(\tp_w))) - \frac{\mathrm{ord}_w(\tp_w)}{\mathrm{ord}_w(q_E)} \log_p(N(N_{K'_w/F_\p}(q_E))) \right) \\
        & = -\frac{1}{2} \left(\log_p(N(N_{K'_w/F_\p}(\tp_w))) - \frac{\mathrm{ord}_w(\tp_w)}{\mathrm{ord}_w(q_E)} [K'_w:F_\p] \log_p(N(q_E)) \right) \\
        & = -\frac{1}{2} \log_{N(q_E)}(N_{K'_w/\qp}(\tp_w)).
    \end{align*}
\end{proof}

\subsection{Proof of Theorem 7.1}
Finally, in this section we prove Theorem \ref{ineq4}.

\begin{thm}
 Let $\chi \in \{1, \epsilon_K\}$ be a quadratic character of $F$ with conductor coprime with $\mathfrak{np}$. Assume that 
 \begin{enumerate}
     \item $\chi(\p) = 1$,
     \item $\mathrm{rank}_\z \hspace{0.01mm} \hspace{1mm} E(K')^\chi = 1$,
     \item $\Varpi(E/K')^\chi_{p^\infty}$ is finite.
 \end{enumerate}
 Then
 \begin{align*}
     X^{cc}_{Gr}(\f/K')^\chi \otimes_\I \I_P \cong \I_P/PI_P.
 \end{align*}
\end{thm}

\begin{proof}
    Let $P \in E(K')^\chi$ be a generator of the free part of $E(K')^\chi$. Fix a lift $P^\dagger = (P,(\tp_w)_{w|\p}) \in E^\dagger(K')^\chi$ of $P$. If $\chi$ is the trivial character, set
    \begin{align*}
        q_\chi := (0,q_E) \in E^\dagger(F) \subset E(F) \times F_\p^\times,
    \end{align*}
    otherwise, set
    \begin{align*}
        q_\chi := (0,(q_E,q_E^{-1})) \in E^\dagger(K)^\chi \subset E(K) \times K_{w_1}^\times \times K_{w_2}^\times,
    \end{align*}
    where $\p O_K = w_1w_2$. Since $\Varpi(E/K')^{\chi}_{p^\infty}$ is finite, Lemma \ref{iedagger is an iso} and the hypothesis $\mathrm{rank}_\z \hspace{0.01mm} \hspace{1mm} E(K')^\chi = 1$ imply that 
    \begin{align*}
        \widetilde{H}^1_f(K',V_f)^\chi \cong (E^\dagger(K') \otimes L)^\chi \cong L \cdot P^\dagger \oplus L \cdot q_\chi.
    \end{align*}
    By $\doteq$, we mean equality up to some nonzero constant multiple. If $K'=F$, using Corollary \ref{PairingInTermsOfLog} we have the following equality 
    \begin{align*}
        \langle q_\chi,P^\dagger \rangle^{Nek}_F \doteq \log_{E,N}(P),
    \end{align*}
    where $\log_{E,N}(\cdot) := \log_{N(q_E)} \circ N \circ \Phi_{Tate}^{-1}(\cdot)$.
    \\
    If $K' = K$ is a quadratic extension of $F$, since $\p$ splits in $K$, write $i_1 : K \hookrightarrow K_{w_1} \cong F_\p$ and $i_2 : K \hookrightarrow K_{w_2} \cong F_\p$. Then $i_2 = i_1 \circ \sigma$, where $\sigma \in Gal(K/F)$ is the non trivial element. Since $P^\dagger \in E^\dagger(K)^\chi$, we have $P^\sigma = -P$ and $\tp_{w_1} = \tp_{w_2}^{-1}$.
    Now clearly, $q_\chi = q_{w_1}-q_{w_2}$, we get again using Corollary \ref{PairingInTermsOfLog} and the fact that $K_{w_j} \cong F_\p$ for both $j = 1,2$,
    \begin{align*}
        \langle q_\chi, P^\dagger \rangle^{Nek}_K & \doteq \log_{N(q_E)}(N_{K_{w_1}/\qp}(\tp_{w_1})) - \log_{N(q_E)}(N_{K_{w_2}/\qp}(\tp_{w_2})) \\
        & \doteq \log_{E,N}(i_1(P)) - log_{E,N}(i_2(P)) \\
        & \doteq 2 \log_{E,N}(P).
    \end{align*}
    Note that these equalities hold upto multiplication by a nonzero constant. Since $\widetilde{H}^1_f(K',V_f)^\chi$ is generated by $P^\dagger$ and $q_\chi$, using Proposition \ref{NonDegeneracyOfPairing} we obtain
    \begin{align*}
        det \langle -,- \rangle^{Nek}_{K,\chi} & = det \begin{pmatrix}
            \langle q_\chi, q_\chi \rangle^{Nek}_K & \langle q_\chi, P^\dagger \rangle^{Nek}_K \\
            \langle P^\dagger,q_\chi \rangle^{Nek}_K & \langle P^\dagger, P^\dagger \rangle^{Nek}_K
        \end{pmatrix} \\
        & \doteq det \begin{pmatrix}
            0 & \log_{E,N}(P) \\
            -\log_{E,N}(P) & 0
        \end{pmatrix} \\
        & \doteq \log_{E,N}^2(P).
    \end{align*}
    Using the fact that the only zeros of $\log_{N(q_E)}$ are the roots of unity and powers of $N(q_E)$ and the description of $N$ in degree 2 extension, we conclude that since $P$ is of infinite order, $\log_{E,N}(P)$ is nonzero and hence so is $det \langle -,- \rangle^{Nek}_{K',\chi}$. In particular, $\langle -,- \rangle^{Nek}_{K',\chi}$ is non-degenerate.
    Hence, using Corollary \ref{NonDegeneracyEquivalence} and the short exact sequence (\ref{ses for MW group}), we get that
    \begin{align*}
        \mathrm{len}_P(X^{cc}_{Gr}(\f/K')^\chi) = \mathrm{dim}_L H^1_f(K',V_f)^\chi = 1.
    \end{align*}
    This implies that $X^{cc}_{Gr}(\f/K')^\chi \otimes \I_P \cong \I_P/P\I_P$.
\end{proof}

\section{Choice of the Field $K$} \label{ChoiceOfK}
In this section, we will make a careful choice of a quadratic extension $K/F$ following \cite{FriedbergHoffstein}.
\\
\begin{lem} \label{Choice of K Lemma}
    Let $\n_E = \mathfrak{np}$ be the conductor of $E/F$ with $\mathfrak{p} \nmid \mathfrak{n}$. Assume that $\mathrm{rank}_\z \hspace{0.01mm} \hspace{1mm} E(F) = 1$ and $\Varpi(E/F)_{p^\infty}$ is finite. Then there exists an imaginary quadratic extension  $K/F$, of discriminant coprime with $\n_E$, such that 
    \begin{enumerate}
        \item $\mathrm{ord}_{s=1} L(E^K/F,s) = 1$, where $E^K$ is the twist of $E$ by the quadratic character associated with $K$,
        \item $\mathrm{rank}_\z \hspace{0.01mm} \hspace{1mm} E(K) = 2$ and $\Varpi(E/K)_{p^\infty}$ is finite.
    \end{enumerate}
\end{lem}

\begin{proof}
    Since $\Varpi(E/F)_{p^\infty}$ is finite and $\mathrm{rank}_\z \hspace{0.01mm} \hspace{1mm} E(F) = 1$, by the $p$-parity result in \cite[Theorem 8.10]{NekParity}, we get that
    \begin{align*}
        w(E/F) = -1,
    \end{align*}
    where $w(E/F)$ is the global root number of $E/F$.
    \\
    Let $\xi$ be a non-trivial quadratic idele class character of $F$ with conductor $c_\xi$ coprime to $\n\p$ such that
    \begin{align*}
        \xi(-1) = 1, \hspace{5mm} \xi(\mathfrak{l}) = 1
    \end{align*}
    for every prime $\mathfrak{l}|\n_ED_F$ and write $E^\xi/F$ for the twist of $E/F$ by $\xi$. Then as in \cite[Section 2]{SignChanges}, we have
    \begin{align*}
        w(E^\xi/F) = \xi(\n_E)w(E/F) = -1.
    \end{align*}
    Now in the notations of \cite{FriedbergHoffstein}, choosing $S_0 = \{\mathfrak{l} : \mathfrak{l}|\n_ED_F\}$, we get that for every $\psi \in \Psi(S_0;\xi)$, $w(E^{\xi\psi}/F) = -1$. So by applying \cite[Theorem B(2)]{FriedbergHoffstein}, we get a quadratic character $\xi'$ of conductor coprime to $c_\xi\n\p$ such that
    \begin{align*}
        \xi'(-1) = -1, \hspace{5mm} \xi'(\mathfrak{l}) = 1
    \end{align*}
    for every $\mathfrak{l}|\n_ED_F$ with the property that
    \begin{align*}
        \mathrm{ord}_{s=1} L(E^{\xi\xi'}/F,s) = 1.
    \end{align*}
    Choose $K$ to be the imaginary quadratic extension of $F$ associated with $\epsilon_K := \xi\xi'$. Since the conductor of $\epsilon_K$ is coprime to $c_\xi\n\p$, there is some finite prime (coprime to $c_\xi\n\p$) which ramifies in $K$. So the hypothesis 4 from Theorem \ref{WanMainResult} is also being satisfied for this $K$.
    \\
    Now by \cite{Zhang}, we get that $\mathrm{rank}_\z \hspace{0.01mm} \hspace{1mm} E(K)^{\epsilon_K} = 1$ and that $\# \left( \Varpi(E/K)^{\epsilon_K} \right) < \infty$. So by the hypothesis, we obtain
    \begin{align*}
        \mathrm{rank}_\z \hspace{0.01mm} E(K) = 2, \hspace{5mm} \# \left( \Varpi(E/K)_{p^\infty} \right) < \infty.
    \end{align*}
\end{proof}

\begin{rem}
    (a) Choose $\n^- = 1$ and $\n^+ = \n$. Then we have that $\n^+$ is divisible by the primes split in $K$ and the requirement that $\n^-$ is divisible by the primes inert in $K$ becomes vacuous. And similarly, the hypothesis that $\orhof$ is ramified at all $v|\n^-$ also becomes vacuous. In particular, the assumptions 1 and 2 in Theorem \ref{WanMainResult} are being satisfied.
    \\
    (b) Since $\chi(\mathfrak{l}) = 1$ for every $\mathfrak{l}|D_F$, assumption 5 in Theorem \ref{WanMainResult} is satisfied.
    \\
    (c) Note that $\p$ splits in $K$.
\end{rem}

\section{Proof of the Main Result} 
In this section, we prove the main result of the article.
\\
Recall from Sections \ref{cc Sel Gp} and \ref{p-adic weight pairing section} that we had the Selmer groups:
    \begin{align*}
        Sel^{S,cc}_{F_\infty}(\f/K) & := ker \left(H^1(G_{K,S_K},\T_\f \otimes_\I \I^*) \rightarrow \prod_{w|\p} H^1(I_w,\T_{\f,w}^- \otimes_\I \I^*) \right), \\
        Sel^{cc}_{Gr}(\f/K) & := ker \left(H^1(G_{K,S_K},\T_\f \otimes_\I \I^*) \rightarrow \prod_{w|\mathfrak{p}} H^1(K_w,\T_{\f,w}^- \otimes_\I \I^*) \right),
    \end{align*}
    where $\T_{\f,w}^- = \T_{\f}^- \cong \I(\ma_\p^* \chi_\Gamma^{-1/2})$ as $\I[G_\p]$-modules.
 
We need the following lemma.

\begin{lem} \label{ineq3}
    $\mathrm{len}_P(X^{S,cc}_{F_\infty}(\f/K)) \leq \mathrm{len}_P(X^{cc}_{Gr}(\f/K)) + 2$.
\end{lem}

\begin{proof}
    Recall we had 
    \begin{align*}
        \T_\f^- \otimes \I^* \cong \mathbb{A}_\f^-.
    \end{align*}
    For simplicity, we write $S_{Gr}$ (resp. $S_{F_\infty}$) for $Sel^{cc}_{Gr}(\f/K)$ (resp. $Sel^{S,cc}_{F_\infty}(\f/K)$ and $X_{Gr}$ (resp. $X_{F_\infty}$) for $X^{cc}_{Gr}(\f/K)$ (resp. $X^{S,cc}_{F_\infty}(\f/K)$). By construction, we have an exact sequence 
    \begin{align*}
        0 \rightarrow S_{Gr} \rightarrow S_{F_\infty} \rightarrow \oplus_{w|\p} H^1(G_w/I_w,(\A_\f^-)^{I_w})
    \end{align*}
    which gives us by taking Pontryagin duals and then taking localization with respect to $P$ the following exact sequence
    \begin{align*}
        \oplus_{w|\p} H^1(G_w/I_w,(\A_\f^-)^{I_w})^*_P \rightarrow (X_{F_\infty})_P \rightarrow (X_{Gr})_P \rightarrow 0.
    \end{align*}
    As $\p$ splits in $K$, we have
    \begin{align*}
        \mathrm{len}_P(X_{F_\infty}) \leq \mathrm{len}_P(X_{Gr}) + 2\mathrm{len}_P(H^1(G_\p/I_\p,(\A_\f^-)^{I_\p})^*).
    \end{align*}
    In particular, it remains to show that $\mathrm{len}_P(H^1(G_\p/I_\p,(\A_\f^-)^{I_\p})^*) = 1$. 
    Since $\T_\f^- \cong \I(\ma_\p^* \chi_\Gamma^{-1/2})$ and $\A_\f^- \cong \T_\f^- \otimes \I^*$, we get
    \begin{align*}
        \A_\f^- \cong \I^*(\ma_\p^*\chi_\Gamma^{-1/2}).
    \end{align*}
    So $(\A_\f^-)^{I_\p} = \A_\f^-[\vp] \cong (\I/\vp\I)^*(\ma_\p^*)$. Applying the functor $H^1(G_\p/I_\p, -)$, we get
    \begin{align*}
        H^1(G_\p/I_\p,(\A_\f^-)^{I_\p}) = H^1(G_\p/I_\p,(\I/\vp\I)^*(\ma_\p^*)) = (\I/\vp\I)^*/(\ma_\p-1).
    \end{align*}
    So 
    \begin{align*}
        H^1(G_\p/I_\p,(\A_\f^-)^{I_\p})^*_P & \cong ((\I/\vp\I)^{**}[\ma_\p-1])_P \\
        & \cong ((\I/\vp\I)[\ma_\p-1])_P \\
        & \cong (\I_P/\vp\I_P)[P(\ma_\p) - 1] \\
        & = \I_P/\vp\I_P
    \end{align*}
    and hence we get that 
    \begin{align*}
        \mathrm{len}_P(H^1(G_\p/I_\p,(\A_\f^-)^{I_\p})^*) = 1.
    \end{align*}
\end{proof}

We are ready to prove Theorem \ref{Our Main Thm}.

\begin{thm} \label{main theorem in text}
    Let $p > 5$ be a prime and $\p = pO_F$, where $O_F$ is the ring of integers of $F$. Let $E/F$ be an elliptic curve having split multiplicative reduction at $\p$. Suppose \textbf{(irred)} and \textbf{(MML)} hold. Assume that
    \begin{enumerate}
        \item $\mathrm{rank}_\z \hspace{0.1mm} E(F) = 1$,
        \item $\Varpi(E/F)_{p^\infty}$ is finite.
    \end{enumerate}
    Then $\mathrm{ord}_{s=1} L(E/F,s) = 1$.
\end{thm}

\begin{proof}
   Since $\mathrm{rank}_\z \hspace{0.01mm} E(F) = 1$ and $\Varpi(E/F)_{p^\infty}$ is finite, we get that $\mathrm{dim}_{\qp} Sel_E(F) = 1$ and hence by \cite[Theorem E]{NekParity} that $w(E/F) = -1$. Let $K$ be the imaginary quadratic extension of $F$ chosen in Lemma \ref{Choice of K Lemma}. Using Corollary \ref{ineq1}, we obtain
   \begin{align} \label{Main proof ineq1}
       4 & \leq \mathrm{ord}_{k=2} L_p^{cc}(f_\infty/K,k).
   \end{align}
   Using Corollary \ref{ineq2} and Lemma \ref{ineq3}, respectively, we get the following inequalities
   \begin{align} \label{Main proof ineq2}
       \mathrm{ord}_{k=2} L_p^{cc}(f_\infty/K,k) \leq \mathrm{len}_P(X^{S,cc}_{F_\infty}(\f/K)) \leq \mathrm{len}_P(X^{cc}_{Gr}(\f/K)) + 2.
   \end{align}
   Finally, by our hypothesis and Lemma \ref{Choice of K Lemma}, we have that
\begin{align*}
    \mathrm{rank}_\z \hspace{0.01mm} E(F) & = 1 \\
    \mathrm{rank}_\z \hspace{0.01mm} E(K)^{\epsilon_K} & = 1,
\end{align*}
where $\epsilon_K$ is the quadratic character attached to $K$ and that both $\Varpi(E/F)_{p^\infty}$ and $\Varpi(E/K)_{p^\infty}^{\epsilon_K}$ are finite. Also, we know that $\p$ splits in $K$. So, the hypotheses of Theorem \ref{ineq4} are being satisfied by both the trivial character and $\epsilon_K$. We get that 
\begin{align*}
    X^{cc}_{Gr}(\f/K)_P & \cong X^{cc}_{Gr}(\f/F)_P \oplus X^{cc}_{Gr}(\f/K)_P^{\epsilon_K} \\
    & \cong \I_P/P\I_P \oplus \I_P/P\I_P
\end{align*}
which gives us
\begin{align} \label{Main proof ineq3}
    \mathrm{len}_P(X^{cc}_{Gr}(\f/K)) + 2 = 4.
\end{align}
Combining (\ref{Main proof ineq1}),(\ref{Main proof ineq2}) and (\ref{Main proof ineq3}), we conclude that 
\begin{align*}
    \mathrm{ord}_{k=2} L_p^{cc}(f_\infty/K,k) = 4
\end{align*}
which, again, by Corollary \ref{ineq1} gives us that 
\begin{align*}
    \mathrm{ord}_{s=1} L(E/K,s) = 2.
\end{align*}
Now, we know that $L(E/K,s) = L(E/F,s) \cdot L(E^K/F,s)$. Further, by Lemma \ref{Choice of K Lemma}, we have that $\mathrm{ord}_{s=1} L(E^K/F,s) = 1$. Using this factorization and the previous equality, we finally get that
\begin{align*}
    \mathrm{ord}_{s=1} L(E/F,s) = 1.
\end{align*}
\end{proof}

\section{A $p$-Converse Result over $\q$} \label{p converse over Q}
In this section, using Theorem \ref{Our Main Thm} we remove the hypothesis ``there exists a prime $q \mid\mid N_E, q \neq p$ such that $p \nmid \mathrm{ord}_q j_E$'' in \cite[Theorem A]{Rodolfo} and establish the following $p$-converse theorem for primes $p > 5$.

\begin{thm} (Theorem \ref{p-converse over Q}) \label{proof over Q}
    Let $p > 5$ be a rational prime. Let $E/\q$ be an elliptic curve of conductor $N_E := Np$. Assume that $E$ has split multiplicative reduction at $p$. Suppose \textbf{(irred)} holds. Further, assume that
    \begin{enumerate}
        \item $\mathrm{rank}_\z \hspace{0.1mm} E(\q) = 1$,
        \item $\Varpi(E/\q)_{p^\infty}$ is finite.
    \end{enumerate}
    Then $\mathrm{ord}_{s=1} \hspace{0.1mm} L(E/\q,s) = 1$.
\end{thm}

\begin{proof}
    Using \cite[Theorem E]{NekParity} and the fact that $\Varpi(E/\q)_{p^\infty}$ is finite, we get that 
    \begin{equation*}
        w(E/\q) = -1.
    \end{equation*}
    Let $\xi$ be a quadratic idele class character of $\q$ which takes the following values
    \begin{equation*}
        \xi(-1) = -1, \hspace{5mm} \xi(q) = 1
    \end{equation*}
    for every prime $q \mid N_E$. Then in the notations of Section \ref{ChoiceOfK},
    \begin{align*}
        w(E^\xi/\q) = -\xi(N_E)w(E/\q) = 1.
    \end{align*}
    Choose $S_0 = \{q : q \mid N\}$. Then for some $\psi \in \Psi(S_0;\xi)$, we have $w(E^{\xi\psi}/\q) = 1$. By \cite[Theorem B(1)]{FriedbergHoffstein}, we can choose a quadratic Dirichlet character $\xi'$ such that
    \begin{equation*}
        \xi'(-1) = -1, \hspace{5mm} \xi'(p) = -1
    \end{equation*}
    and
    \begin{equation*}
        L(E^{\xi\xi'}/\q,1) \neq 0.
    \end{equation*}
    Choose the field $F$ to be the real quadratic field corresponding to the character $\xi\xi'$ and write $E^F$ for the elliptic curve $E^{\xi\xi'}$. Then using \cite{Kolyvagin} and \cite{GrossZagier}, we have 
    \begin{equation*}
        \mathrm{rank}_\z \hspace{0.1mm} E^F(\q) = 0, \hspace{5mm} \# \left( \Varpi(E^F/\q)_{p^\infty} \right) < \infty.
    \end{equation*}
    In particular, we obtain
    \begin{align*}
        \mathrm{rank}_\z \hspace{0.1mm} E(F) = 1, \hspace{5mm} \# \left( \Varpi(E/F)_{p^\infty} \right) < \infty.
    \end{align*}
    Moreover, a minimal modular lifting of the residual representation of $E/F$ can be obtained by a minimal modular lifting over $\q$. Also, note that by our choice of $\xi$ and $\xi'$, $p$ is inert in $F$. Thus hypotheses of Theorem \ref{Our Main Thm} are satisfied for $E/F$ and 
    applying the same theorem, we get that $\mathrm{ord}_{s=1} \hspace{0.1mm} L(E/F,s) = 1$ and hence, 
    \begin{equation*}
        \mathrm{ord}_{s=1} \hspace{0.1mm} L(E/\q,s) = 1.
    \end{equation*}
    This completes the proof of the theorem.\end{proof}

\begin{rem}[Possible Extensions]\label{possible-extensions}
    \begin{enumerate}
        \item Let $K$ be an imaginary quadratic field and $p >5 $ be an integer prime which splits in $K$. Let $E/\q$ be an elliptic curve with split multiplicative reduction at $p$.  If there is a rank $0$ $p$-converse result over $\q$, removing the assumption $(ii)$ in  \cite[Theorem C]{SkinnerRank0}, then under the assumptions of Theorem \ref{p-converse over Q}, one can establish a rank $1$ $p$-converse result for the base change of $E$ to $K$.

        \item Let $F$ be a real quadratic field and $p >5 $ be an integer prime inert in $F$. Let $E$ be an elliptic curve with split multiplicative reduction at the prime $\p$ dividing $p$ in $F$.  If there is a rank $0$ $p$-converse result over $F$ in this split multiplicative reduction setting, then using Theorem \ref{Our Main Thm}, under suitable hypotheses, one should be able to establish a rank $1$ $p$-converse result for base change of $E$ to  certain biquadratic extensions of $\q$.
    \end{enumerate}
\end{rem}

\section{Example}

   We give an example of a triple $(E,F,p)$ which satisfies all the assumptions of Theorem \ref{Our Main Thm}.
    \\
    Consider the quadratic real field $F = \q(\sqrt{2})$. Then the rational prime $p=11$ is inert in $F$. Write $\p = (11)O_F$. Now consider the elliptic curve $X_0(11)$ which is given by the Weierstrass equation:
    \begin{align*}
        E : y^2+y=x^3-x^2-10x-20.
    \end{align*}
    Note that $\mathrm{rank}_\z E(\q) = 0$. Using \cite{lmfdb} (https://www.lmfdb.org/EllipticCurve/2.2.8.1/121.1/a/2), one can see that $E/F$ has split multiplicative reduction at $\p$. Furthermore, \cite{lmfdb} indicates that the rank of the Mordell-Weil group  $E(F)$ is 1, and the Tate-Shafarevich group $\Varpi(E/F)$ is finite. Moreover, we have, again from \cite{lmfdb}, that $\overline{\rho}_\f \cong \overline{\rho}_{E,p}$ has maximal image and hence, \textbf{(irred)} is being satisfied. Finally, $E/F$ is the base change of $X_0(11)$ to F, so a minimal modular lifting of the residual representation of $E/\q$ gives us the required minimal modular lifting, so \textbf{(MML)} is being satisfied. Hence, all the hypotheses of Theorem \ref{Our Main Thm} are met.

\bibliographystyle{amsalpha}
\bibliography{p_converse}

\end{document}